\documentclass[10 pt]{amsart}

\usepackage{amsmath}
\usepackage{amssymb}
\usepackage{amsfonts}
\usepackage{mathrsfs}
\usepackage[OT2,T1]{fontenc}
\usepackage{pst-node}
\usepackage{tikz}
\usetikzlibrary{matrix,arrows}
\usepackage{tikz-cd}
\usepackage{enumerate}

\mathchardef\ordinarycolon\mathcode`\:
\mathcode`\:=\string"8000
\begingroup \catcode`\:=\active
  \gdef:{\mathrel{\mathop\ordinarycolon}}
\endgroup

\DeclareSymbolFont{cyrletters}{OT2}{wncyr}{m}{n}
\DeclareMathSymbol{\Sha}{\mathalpha}{cyrletters}{"58}

\newtheorem{thm}{Theorem}[section]
\newtheorem{prop}[thm]{Proposition}
\newtheorem{lem}[thm]{Lemma}
\newtheorem{cor}[thm]{Corollary}

\newtheorem{defn}[thm]{Definition}

\newtheorem{thm*}{Theorem}

\renewcommand{\bold}{\boldsymbol}
\newcommand{\mf}{\mathfrak}

\renewcommand{\a}{\alpha}

\renewcommand{\c}{\gamma}
\newcommand{\G}{\Gamma}
\renewcommand{\d}{\delta}
\newcommand{\D}{\Delta}

\renewcommand{\l}{\lambda}

\newcommand{\m}{\mu}

\newcommand{\s}{\sigma}
\renewcommand{\t}{\tau}
\newcommand{\w}{\omega}

\newcommand{\p}{\mf{p}}
\newcommand{\q}{\mf{q}}

\newcommand{\mc}{\mathcal}
\newcommand{\mb}{\mathbb}
\newcommand{\mr}{\mathrm}
\newcommand{\ms}{\mathscr}

\newcommand{\leqs}{\leqslant}
\newcommand{\geqs}{\geqslant}
\newcommand{\isom}{\cong}
\newcommand{\sq}{\sqrt}
\renewcommand{\i}{\infty}

\newcommand{\ovl}{\overline}

\renewcommand{\iff}{\Leftrightarrow}

\newcommand{\mto}{\mapsto}
\newcommand{\inj}{\hookrightarrow}

\renewcommand{\sb}{\subset}

\newcommand{\xto}{\xrightarrow}

\renewcommand{\O}{\Omega}

\newcommand{\Q}{\mb{Q}}
\newcommand{\Z}{\mb{Z}}

\newcommand{\C}{\mb{C}}

\renewcommand{\o}{\mc{O}}

\DeclareMathOperator{\Gal}{Gal}

\DeclareMathOperator{\coker}{coker}

\DeclareMathOperator{\Frob}{Frob}
\DeclareMathOperator{\res}{res}

\DeclareMathOperator{\rank}{rank}

\DeclareMathOperator{\ord}{ord}

\DeclareMathOperator{\Hom}{Hom}
\DeclareMathOperator{\End}{End}

\DeclareMathOperator{\N}{N}
\DeclareSymbolFont{cyrletters}{OT2}{wncyr}{m}{n}
\DeclareMathSymbol{\Sha}{\mathalpha}{cyrletters}{"58}

\title[On the main conjecture of Iwasawa theory]
  {On the main conjecture of Iwasawa theory for certain non-cyclotomic $\mathbb{Z}_p$-extensions}
\date{}
\author{YUKAKO KEZUKA}
\thanks{This research was partially supported by the SFB 1085 ``Higher invariants'' at the University of Regensburg, funded by the Deutsche Forschungsgemeinschaft (DFG)}

\begin{document}

\maketitle

\begin{abstract}Let $K=\mathbb{Q}(\sqrt{-q})$, where $q$ is any prime number congruent to $7$ modulo $8$, with ring of integers $\o$ and Hilbert class field $H$. Suppose $p\nmid [H:K]$ is a prime number which splits in $K$, say $p\o=\mathfrak{p}\mathfrak{p}^*$. Let $H_\infty=HK_\infty$ where $K_\infty$ is the unique $\mathbb{Z}_p$-extension of $K$ unramified outside $\mathfrak{p}$. Write $M(H_\infty)$ for the maximal abelian $p$-extension of $H_\infty$ unramified outside the primes above $\p$, and set $X(H_\infty)=\Gal(M(H_\infty)/H_\infty)$. In this paper, we establish the main conjecture of Iwasawa theory for the Iwasawa module $X(H_\infty)$. As a consequence, we have that if $X(H_\infty)=0$, the relevant $L$-values are $\mf{p}$-adic units. In addition, the main conjecture for $X(H_\infty)$ has implications toward (a) the BSD Conjecture for a class of CM elliptic curves; (b) weak $\p$-adic Leopoldt conjecture.
\end{abstract}

\section{Introduction}

The study of the main conjectures of Iwasawa theory, which relates the $p$-adic $L$-functions to the characteristic power series of certain Iwasawa modules, has been very influential in the development of modern number theory, and has been applied to a wide circle of problems in which values of $L$-functions play a key role. It provided, for example, one of the most fruitful approaches to understanding the conjecture of Birch and Swinnerton-Dyer. The strongest general result known in this direction so far is for elliptic curves $E$ defined over an imaginary quadratic field $K$ with complex multiplication by the ring of integers of $K$. For these curves, Rubin showed \cite[Theorem 11.1]{rubin}, by first proving appropriate main conjectures, that if $L(E/K,1)\neq 0$, then the $p$-part of the Birch--Swinnerton-Dyer conjecture holds for all $p$ not dividing the number of roots of unity in $K$. In particular, the prime $p=2$ is always omitted, and in many ways this is the most interesting prime, because the product of Tamagawa factors which appears in the formula of the Birch--Swinnerton-Dyer conjecture can be divisible by a large power of $2$. However, the arguments used in the proof seem very difficult to extend to cover this case, except when $E$ has potential ordinary reduction at the primes above $2$. In this paper, we study a family of quadratic twists of elliptic curves with complex multiplication which are no longer defined over $K$, but over the Hilbert class field $H$ of $K$. We formulate and prove certain main conjectures at primes including $p=2$.

 Let $K=\mb{Q}(\sq{-q})$, where $q$ is a prime congruent to $7$ modulo $8$. Then the discriminant of $K$ is equal to $-q$, so the class number $h$ of $K$ is odd by genus theory. Let $\o$ denote the ring of integers of $K$, and let $H$ be the Hilbert class field of $K$. Let $p$ denote a prime such that $(p,q)=1$ and $p$ splits in $K$, say $p\o=\mf{p}\mf{p}^*$. Write $H_\infty=HK_\infty$ where $K_\infty$ is the unique $\Z_p$-extension of $K$ unramified outside $\mf{p}$, given by global class field theory. Let $\mathscr{G}=\Gal(H_\infty/K)$ and $G=\Gal(H_\infty/K_\infty)$, and assume $(p,h)=1$. Denote by $M(H_\infty)$ the maximal abelian $p$-extension of $H_\infty$ unramified outside the primes of $H_\infty$ above $\p$, and write
\[X(H_\infty)=\mathrm{Gal}(M(H_\infty)/H_\infty).\]
Then $X(H_\infty)$ is a finitely generated torsion module over the Iwasawa algebra $\Z_p[[\ms{G}]]$. Let $\ms{I}$ be the ring of integers of the completion of the maximal unramified extension of $K_\mf{p}$. 
We briefly recall the structure theorem for finitely generated torsion modules over the Iwasawa algebra $\Lambda_\ms{I}(\ms{G})=\ms{I}[[\ms{G}]]$ of $\ms{G}$ with coefficients in $\ms{I}$.  Given a finitely generated torsion $\Lambda_\ms{I}(\ms{G})$-module $M$ and $\chi\in G^*:=\Hom(G,\mathbb{C}_p^\times)$, write $M^\chi$ for the largest submodule of $M$ on which $G$ acts via $\chi$. Since $p\nmid \#(G)$ by assumption, any $\Lambda_\ms{I}(\ms{G})$-module decomposes into the direct sum of its $\chi$-components. Furthermore, $\Lambda_\ms{I}(\ms{G})^\chi$ is (non-canonically) isomorphic to the ring $\ms{I}[[T]]$ of formal power series in indeterminate $T$ with coefficients $\ms{I}$. Thus, the well-known structure theorem for finitely generated torsion $\ms{I}[[T]]$-modules easily implies that there exist elements $f_1,\ldots ,f_r$ of $\Lambda_\ms{I}(\ms{G})$ and pseudo-isomorphisms
$$\oplus_{j=1}^r \Lambda_\ms{I}(\ms{G})/(f_i)\to M \text{\;\;\; and \;\;\;}M\to \oplus_{j=1}^r \Lambda_\ms{I}(\ms{G})/(f_i).$$
The ideal $(\prod_{i=1}^r f_i)\Lambda_\ms{I}(\ms{G})$ is an invariant of $M$ called the characteristic ideal of $M$, and is denoted by $\mathrm{char}(M)$. Furthermore, for every $\chi$, we will denote by $\mathrm{char}\left(M^\chi \right)\sb \Lambda_\ms{I}(\ms{G})^\chi$ the characteristic ideal of the $\Lambda_\ms{I}(\ms{G})^\chi$-module $M^\chi$. Similarly, given a finitely generated torsion $\Z_p[[\ms{G}]]$-module $X$, we will denote by $\mathrm{char}\left(X\right)$ and $\mathrm{char}\left(X^\chi\right)$ the characteristic ideals of $X\hat{\otimes}_{\Z_p}\ms{I}$ and $(X\hat{\otimes}_{\Z_p}\ms{I})^\chi$, respectively. 

In this paper, we shall prove the main conjecture of Iwasawa theory for the extension $H_\infty/H$. We remark that the case $p=2$ is specifically excluded from \cite[Chapter III]{dS} where de Shalit formulates the main conjecture and gives some evidence in favour of it. We construct in Section \ref{section4.2} an $\ms{I}$-valued pseudo-measure $\nu_\mathfrak{p}$ on $\mathscr{G}$ which interpolates relevant Hecke $L$-values. For a precise statement of this result, see Theorem \ref{thm4.2.7}. Define $\bold{\varphi}=I_\ms{I}(\mathscr{G})\nu_\mathfrak{p}$, where $I_\ms{I}(\mathscr{G})$ denotes the augmentation ideal of $\Lambda_\ms{I}(\ms{G})$. We can now state the main result of this paper.

\begin{thm*}[Theorem \ref{mc}, Main Conjecture for $H_\infty/H$]
For every $\chi\in G^*$, we have
\[\mathrm{char}\left(X(H_\infty)^\chi\right)=\bold{\varphi}^\chi.\]
\end{thm*}

We end this section by outlining the proof and applications of Theorem \ref{mc}. The proof closely follows the methods of Rubin \cite{rubin}. Later, Gonzalez-Avil\'{e}s extended Rubin's methods to study the case $p=2$ \cite[Theorem 3.5.1]{gon} using quadratic twists of the elliptic curve $X_0(49)$ which is defined over $\Q$ and has complex multiplication by the ring of integers of $\Q(\sqrt{-7})$. In the proof, it is vitally important that $2$ is a potentially ordinary prime for all quadratic twists of $X_0(49)$, since $2$ splits in $\Q(\sqrt{-7})$. Unfortunately, there are no other elliptic curves with complex multiplication defined over $\Q$ for which $2$ is a potentially ordinary prime, since $\Q(\sqrt{-7})$ is the only imaginary quadratic field with class number one in which $2$ splits. We will thus extend their ideas by looking at a family of elliptic curves which are no longer defined over $\Q$ or $K$. In \cite{gr1}, Gross proved the existence of an elliptic curve $A(q)$ defined over the field $J$ of index $2$ in $H$ with complex multiplication by $\o$ and minimal discriminant $-q^3$. In the case $q=7$, we have $A(7)=X_0(49)$. Let $E$ be any quadratic twist of $A(q)$ by a quadratic extension of the form $H(\sq{\lambda})/H$ of discriminant prime to $2q$, $\lambda\in K^\times$. Let $p$ be a prime such that $E$ has good reduction at all places of $H$ above $p$, and $p$ splits in $K$, say $p\o=\mf{p}\mf{p}^*$. The theory of complex multiplication then tells us that the reduction of $E$ is of ordinary type at all primes of $H$ above $p$, a fact which will be important for the arguments of Iwasawa theory to follow. In particular, $p=2$ satisfies these conditions, and we will pay special attention to this case which is quite different in nature. Note that the ideals in $\o$ are in not in general principal since the class number of $K$ is not equal to $1$ if $q>7$. Given a non-zero ideal $\mf{a}$ of $\o$, define $E_\mf{a}=\cap_{\alpha\in \mf{a}} E_\alpha$, where $E_\alpha=\ker \left(E(\bar{F})\xto{\alpha}E(\bar{F}))\right)$. Let $F_n=H(E_{\mf{p}^n})$, and set
\[F_\i=H(E_{\p^\i}), \;\;\; \mf{H}=\Gal(F_\i/H).\]
Let $\o_\mf{p}$ be the ring of integers of $K_\mf{p}=\Q_p$. We make an observation that in the case $q>7$, we do not have $H=K$, so that the formal group $\widehat{E}$ of $E$ at a place $v$ of $H$ where $E$ has good reduction is not in general a Lubin--Tate group of $E$ over $H_v$. We overcome this difficulty by an argument which involves relative Lubin--Tate groups introduced by de~Shalit \cite[Chapter I \S 1]{dS}, showing that there is a canonical isomorphism $\chi_\mf{p}: \mf{H}\to \o_\mf{p}^\times$ given by the action of $\mf{H}$ on $E_{\mf{p}^\infty}$, and
\[\mf{H}=\Delta\times \Gamma,\]
where $\Delta$ is cyclic of order $p-1$ or $2$, if $p>2$ or $p=2$, and $\Gamma$ is isomorphic to $\o_\mf{p}$. 

Write $\bar{\mc{E}}_{H_\infty}$ and $U_{H_\infty}$ for the groups of global units and principal semi-local units defined in Section \ref{section4.4}. Let $A(H_\infty)$ be the projective limit of the $p$-primary part of the ideal class group of $H_n=F_n\cap H_\infty$ with respect to the norm maps, and let $\bar{\mc{C}}_{H_\infty}$ be the group of elliptic units defined in Section \ref{section4.3}.  Global class field theory provides an exact sequence of $\Z_p[[\ms{G}]]$-modules
\begin{equation}\label{eq4.2}
0\to \bar{\mc{E}}_{H_\infty}/\bar{\mc{C}}_{H_\infty}\to U_{H_\infty}/\bar{\mc{C}}_{H_\infty}\to X(H_\infty)\to A(H_\infty)\to 0.
\end{equation} 
We prove in Chapter \ref{ch6} that 
\[\mathrm{char}\left((U_{H_\infty}/\bar{\mc{C}}_{H_\infty})^\chi\right)=\bold{\varphi}^\chi\]
for every $\chi\in G^*$. In Chapter \ref{ch5}, we construct an Euler system of the elliptic units $\bar{\mc{C}}_{H_\infty}$. Note that in \cite[Chapter II]{dS}, de Shalit sketches rather elaborate proofs of the functional equations that are satisfied by the elliptic function $\Theta$ (see \cite[II.2.3]{dS}) in order to study the properties of the elliptic units, referring back to arguments of Robert and Kubert--Lang. In particular, in \cite[Chapter III]{dS} de Shalit takes a $12$th root of $\Theta$ and shows a norm compatibility relation satisfied by the elliptic units, modulo roots of unity. He then uses that $w_\p$, the number of roots of unity in $K$ congruent to $1$ modulo $\p$, is equal to $1$ for $p>2$. In this paper, we will deal with a $12$th root of $\Theta$ from the beginning, and all the functional equations will be proved directly by simple, purely algebraic arguments with rational functions on $E$, in the spirit of the Appendix of \cite{coa}. 

We then use a variant of \v{C}ebotarev's theorem and induction to establish a divisibility relation between the characteristic ideal of $\left(\bar{\mc{E}}_{H_\infty}/\bar{\mc{C}}_{H_\infty}\right)^\chi$ and that of $A(H_\infty)^\chi$ in $\Z_p[[\Gamma]]$. Since the characteristic ideals of a $\Gamma$-module behave well under extension of scalars, this implies the following divisibility relation in $\Lambda_\ms{I}(\Gamma)=\ms{I}[[\Gamma]]$:
\begin{thm*}[Theorem \ref{thm4.1}]
\begin{enumerate}
\item If $p$ is an odd prime, we have
 \[\mathrm{char}\left(X(H_\infty)^\chi\right)\mid \mathrm{char}\left((U_{H_\infty}/\bar{\mc{C}}_{H_\infty})^\chi\right).\]
\item If $p=2$, we have
 \[\mathrm{char}\left(X(H_\infty)^\chi\right)\mid \bold{\pi}^{k} \mathrm{char}\left((U_{H_\infty}/\bar{\mc{C}}_{H_\infty})^\chi\right)\]
for some integer $k\geqs 0$, where $\bold{\pi}$ is a uniformiser of $\ms{I}$.
\end{enumerate}
 \end{thm*}
The uniformizer $\bold{\pi}$ in part (ii) of the above theorem appears due to the fact that the restriction map in Lemma \ref{lem6} is not necessarily injective for $p=2$.

 In Chapter \ref{ch6}, we finish the proof of the main conjecture by showing that $X(H_\infty)$ and $U_{H_\infty}/\bar{\mc{C}}_{H_\infty}$ have the same Iwasawa invariants. We first follow the paper of Coates and Wiles \cite{coa-wil2} to compute the Iwasawa invariants of $X(H_\infty)$, and then compute the Iwasawa invariants of $U_{H_\infty}/\bar{\mc{C}}_{H_\infty}$ using the analytic class number formula and Kronecker's second limit formula. 

As a corollary, Theorem \ref{mc} implies that $X(H_\infty)=0$ if and only if $\bold{\varphi}$ is a unit. It is known that (see \cite[Section 5]{ckl}), when $p=2$, $X(H_\infty)=0$ for all primes $q\equiv 7\bmod 8$ with $q<500$ and $q\neq 431$. Thus Theorem \ref{thm4.2.7} gives us information on the $\mf{p}$-adic valuation of $\frac{L(\ovl{\psi}_{E/H}^k,k)}{\O_\i(E/H)^{k}}$ for positive integers $k$ with $k\equiv 0\bmod \#(\Delta)$, where $\psi_{E/H}$ denote the Gr\"{o}ssencharacter of $E/H$ and $\O_\i(E/H)$ denotes the complex period. In particular, it tells us in this case that $\bold{\varphi}$ is a unit if and only if \mbox{$(\chi_\mf{p}(\gamma)^2 -1)L(\ovl{\psi}_{E/H}^2,2)/\O_\i(E/H)^2$} is a unit at $\mf{p}$, where $\gamma$ is a topological generator of $\Gamma$. It follows that $X(H_\infty)=0$ if and only if $\ord_\mathfrak{p}\left(\frac{L(\ovl{\psi}_{E/H}^2,2)}{\O_\i(E/H)^{2}}\right)=-1$ or $-3$, if $p>2$ or $p=2$. This can be checked numerically in the case $E=X_0(49)$ and $p=2$. Indeed, we can compute using the computer package Magma that $L(\ovl{\psi}_{E/H}^2,2)/\O_\i(E/H)^2=\frac{1}{8}$, and we have $\ord_\mf{p}\left(\chi_\mf{p}(\gamma)^2-1\right)=3$ since $\gamma$ is a topological generator of $\Gamma\simeq 1+4\o_\mf{p}$. 
We remark also that the main conjecture for $H_\infty/H$ is an important step to proving the main conjecture for $F_\infty/F$, which can be used to study the $p$-part of the Birch--Swinnerton-Dyer Conjecture for $E/H$. Furthermore, for $p=2$, the construction of the $\mf{p}$-adic $L$-function in Chapter \ref{ch4} and the computation of the Iwasawa invariants in Chapter \ref{ch6} of this paper are crucial for the proof in  \cite{ckl} that $X(H_\infty)$ is a finitely generated $\mathbb{Z}_p$-module (the case $p>2$ was proven independently by Gillard \cite[Theorem 3.4]{Gi} and Schneps \cite[Theorem IV]{Sch}). This can be applied to prove the weak $\mf{p}$-adic Leopoldt conjecture for certain non-abelian extensions. Recall that the Leopoldt conjecture holds for a number field $F$ if the Leopoldt defect of $F$ vanishes, or equivalently, if the $p$-adic regulator $R_{p}(F)$ of $F$ (see \eqref{regulator} for a definition with respect to $\mf{p}$) does not vanish. For abelian extensions of $K$,  this follows from Baker's theorem \cite{ba} on linear forms in the $p$-adic logarithms of algebraic
numbers, which was proven by Brumer \cite[Theorem 1]{br}. Now, for $p=2$, it can be shown by an argument involving Nakayama's lemma that $X(\mathcal{F}_\infty)$ is a finitely generated $\Z_2$-module for any quadratic extension $\mathcal{F}$ of $H$ and $\mathcal{F}_\infty=\mathcal{F}H_\infty$. This allows us to prove the weak $\mathfrak{p}$-adic Leopoldt conjecture for a class of non-abelian extensions $\mathcal{F}_\infty/K$, which asserts that the $\p$-adic defect of Leopoldt
is always bounded as one goes up the $\Z_p$-extension $\mathcal{F}_\infty/\mathcal{F}$, or equivalently, that  $X(\mathcal{F}_\infty)$ is $\mathbb{Z}_p[[\Gamma]]$-torsion (see \cite[Lemma 14]{coa2}). As a consequence, we obtain that $E(H_\infty)$ and $E(\mathcal{F}_\infty)$ modulo torsion are finitely generated abelian groups. This is discussed further in \cite{ckl}.

\section{The Gross Curves}\label{ch3}

Take $q$ to be any prime number with $q\equiv 7 \bmod 8$. Let $K=\Q(\sq{-q})$, and fix an embedding $K\inj \C$. Let $E$ be an elliptic curve over $\C$ with $\End_\C(E)=\o$, the ring of integers of $K$. Since $K$ has prime discriminant, the class number, which we denote by $h$, is odd. In the case $q=7$, we can take $E$ to be any quadratic twist of the elliptic curve $A=X_0(49)$ with equation
\[A: y^2+xy=x^3-x^2-2x-1.\]
This is the only family of quadratic twists of elliptic curves with complex multiplication defined over $\Q$ for which $2$ is a potentially ordinary prime, since $q=7$ is the only case in which $K$ has class number one. In general, the theory of complex multiplication tells us that the modular invariant $j(\o)$ is a real number which satisfies an irreducible equation of degree $h$ over $K$, and the Hilbert class field $H$ of $K$ is given by $H=K(j(\o))$.  Furthermore, given a rational prime $p$, $E$ has potentially good ordinary reduction at all primes of $H$ above $p$ if and only if $p$ splits in $K$.

From now on, let $p$ be a prime number such that $E$ has good ordinary reduction at all primes of $H$ above $p$, and $p$ splits in $K$, say $p\o=\p\p^*$. We also define $J=\Q(j(\o))$, which satisfies $[H:J]=2$. Then for any prime number $q$ with $q\equiv 7\bmod 8$, Gross showed in \cite[Theorem 12.2.1]{gr1} that there exists an elliptic curve $A(q)$ which is defined over $J$ with $\End_H(E)=\o$. In the simplest case $q=7$, we have $A(7)=X_0(49)$. This is done by constructing a Gr\"{o}ssencharacter $\psi_q$ of $H$. Let $\mf{a}$ be an integral ideal of $H$. Define $\psi_q$ to be the unique  Gr\"{o}ssencharacter with conductor $(\sq{-q})$ such that, if $\mf{a}$ is an integral ideal of $H$ with $(\mf{a},q)=1$, then
\[\psi_q(\mf{a})=\a,\]
where $\a$ is the unique generator of the principal ideal $\mr{N}_{H/K}(\mf{a})$ which is a square in $\o/\sq{-q}\o$. In particular, we have $\s(\psi_q)=\psi_q$ for all $\s\in \Gal(H/\Q)$.
This defines an isogeny class of elliptic curves defined over $H$ with Gr\"{o}ssencharacter $\psi_q$, $j$-invariant equal to $j(\o)$ and complex multiplication by $\o$. Gross showed that we can pick out  in this isogeny class a unique elliptic curve $A(q)$ defined over $J$ with Gr\"{o}ssencharacter $\psi_{A(q)/H}=\psi_q$ such that $\End_H(A(q))=\o$, $j(A(q))=j(\o)$ and the minimal discriminant ideal is equal to $(-q^3)$. In addition, Gross found an explicit equation for $A(q)$ over $J$. We will show this via a slightly different method. Let us consider a generalised Weierstrass equation of $A(q)$ of the form 
\[y^2+a_1xy+a_3y=x^3+a_2x^2+a_4x+a_6\]
with $a_i\in H$. Let $\D(A(q))$ denote the discriminant for this equation. We claim that we can have $a_i\in J$ with $\D(A(q))=-q^3$. Given an integral ideal $\mf{a}$ of $\o$, let $\s_\mf{a}$ denote the image of $\mf{a}$ via the Artin isomorphism from the ideal class group of $K$ to $G=\Gal(H/K)$, and let $\lambda(\mf{a})$ denote the unique isogeny from $A(q)$ to $B=A(q)^{\s_\mf{a}}$ of degree $\mr{N}\mf{a}$ defined over $H$, characterised by 
\[\lambda(\mf{a})(\mf{u})=\s_\mf{a}(\mf{u})\]
for any $\mf{u}\in A(q)[\mf{c}]$ with $(\mf{c}, \mf{a})=1$.
Let $x'$, $y'$ be the coordinates of any generalised Weierstrass equation for $B$, and let $\D(B)$ be the discriminant of this equation. We write
\[\w_{A(q)}=\frac{dx}{2y+a_1x+a_3}, \;\; \w_{B}=\frac{dx'}{2y'+a_1'x'+a_3'}\]
for the N\'{e}ron differentials. Then we see that the value $\Lambda(\mf{a})\in H^\times$ defined by
\[\lambda(\mf{a})^*(\w_{B})=\Lambda(\mf{a})\w_{A(q)}\]
is such that $\D(B)\Lambda(\mf{a})^{12}$ is independent of the choice of Weierstrass equation for $B$. Further, it is shown in \cite[Appendix, Theorem 8]{coa} that there exists a unique $c_{A(q)}(\mf{a})\in H^\times$ such that $c_{A(q)}(\mf{a})$ gives a canonical $12$th root in $H$ of
\[\frac{\D(A(q))^{\deg \lambda(\mf{a})}}{\D(B)\Lambda(\mf{a})^{12}}=\frac{\D(A(q))^{\mathrm{N}\mf{a}-1}}{\Lambda(\mf{a})^{12}}.\]
Taking appropriate values for $\mf{a}$, we see in particular that $\D(A(q))$ has a $6$th root in $H$. By the definition of $j$ (see \cite[\S 1]{gr1}), $j(A(q))$ has a cube root in $H$ and $j(A(q))-1728$ has a square root in $H$. Note that the only roots of unity in $H$ are $\pm1$, so $j(A(q))$ in fact has a cube root in $J$. Now we have the following.

\begin{prop} The curve $A(q)$ has a model over $J$
\begin{align}\label{eq2}y^2=x^3+\frac{mq}{2^4\cdot 3}x-&\frac{nq^2}{2^5\cdot 3^3}\;\;\;\;\text{ where}\\
 m^3=j(A(q)) \;\;\text{ and }\;\;&n^2=\frac{j(A(q))-1728}{-q},\nonumber
\end{align}
with discriminant equal to $-q^3$. Here, we take the positive square root for $n$.
\end{prop}

\begin{proof} The arguments above show that $m\in J$, and $n\in H$. But $j(A(q))-1728$ and $-q$ are both negative, so $n\in J$ as well. An easy computation then shows that indeed the curve defined by equation \eqref{eq2} has discriminant $-q^3$ and $j$--invariant equal to $j(A(q))$. Now, \cite[Proposition 3.5]{gr2} shows that there is an isomorphism over $J$ from this curve to $A(q)$. 
\end{proof}

The coefficients of the \eqref{eq2} are integral in $J$, except perhaps  at $2$ and $3$. It is not known in general how to write a global minimal equation for $A(q)$ over $J$ explicitly for $q>7$, although Gross has shown that it exists over $J$ (see \cite[Proposition 3.2]{gr2}). Using a classical $2$-descent, Gross showed that, for all $q\equiv 7\bmod 8$, we have \cite[Theorem 22.4.1]{gr1}:
\[A(q)(J)=\Z/2\Z, \;\; A(q)(H)=\Z/2\Z\times \Z/2\Z, \;\; \Sha(A(q)/J)(2)^G=0.\]
There is one additional property of the curves $A(q)$ which is important in carrying out arguments of Iwasawa theory for them. Let $A(q)_{\text{tor}}$ denote the torsion subgroup of $A(q)(\ovl{J})$. By the theory of complex multiplication, $H(A(q)_{\text{tor}})$ is an abelian extension of $H$. In fact, more is true since $\psi_{A(q)/H}$ satisfies $\psi_{A(q)/H}=\varphi_{A(q)}\circ \mr{N}_{H/K}$ where $\varphi_{A(q)}$ is a Gr\"{o}ssencharacter of $K$ with conductor $(\sq{-q})$, and a theorem of Shimura \cite[Theorem 7.44 ]{Sh} states that the existence of such a $\varphi_{A(q)}$ is equivalent to $H(A(q)_{\text{tor}})$ being an abelian extension of $K$.

In what follows, we assume $E$ is a quadratic twist of $A(q)$ by a quadratic extension of $H$ of the form $H(\sq{\l})$, where $\l$ is some non-zero element of $K$ and the discriminant of $H(\sq{\l})/H$ is prime to $2q$. Thus, in particular, $E$ has good ordinary reduction at the primes of $H$ above $2$. 

\begin{prop}\label{thm4} We have 
\begin{align*}\psi_{E/H}&=\varphi_{K}\circ \mr{N}_{H/K},
\end{align*}
where $\varphi_{K}$ is a Gr\"{o}ssencharacter of $K$.
 \end{prop}

\begin{proof} We have remarked that $\psi_{A(q)/H}=\varphi_{A(q)}\circ \mr{N}_{H/K}$. Now, $E$ is a twist of $A(q)$ by a quadratic extension $\mc{M}$ of $H$ which we assumed to be of the form $HM$ where $M$ is a quadratic extension of $K$. Let $\chi_\mc{M}$ (resp. $\chi_M$) be the quadratic character of $H$ (resp. $K$) defining $\mc{M}$ (resp $M$). Then we have $\chi_{\mc{M}}=\chi_M\circ \mathrm{N}_{H/K}$ by class field theory. Now, since $\mc{M}/H$ has discriminant prime to $p$, we have $\psi_{E/H}=\psi_{A(q)/H}\chi_\mc{M}$.  It follows that we can take $\varphi_{K}=\varphi_{A(q)}\chi_M$.
\end{proof} 

Applying \cite[Theorem 7.44]{Sh}, we immediately obtain that the field $H(E_{\text{tor}})$ is abelian over $K$. It also follows that $E$ is isogeneous over $H$ to all of its conjugates under $G$, since we have $\psi_{E/H}=\psi_{E^\s/H}$, that is, $E$ and $E^\s$ have isomorphic Galois representations on their Tate modules, and so, they are isogenous over $H$ by Faltings' theorem.

Given an ideal $\mf{b}$ of $\o$ prime to the conductor $\mf{g}$ of the Gr\"{o}ssencharacter $\varphi_K$, let $\sigma_\mf{b}$ be the Artin symbol of $\mf{b}$ for $H/K$. Let
\[\lambda_{E}(\mf{b}): E\to E^{\sigma_\mf{b}}\]
denote the unique $H$-isogeny whose kernel is $E_\mf{b}$, obtained by restricting the Serre--Tate character of the abelian variety $B/K$ \cite[Theorem 10]{ser-tat}, which is the restriction of scalars of $E$ from $H$ to $K$. If $v$ is any place of $H$, we write $H_v$ for the completion of $H$ at $v$, and write $\o_v$ for its ring of integers. In the case $q=7$, we have $H=K$, so that for every place $v$ of $H$ where $E$ has good reduction, the formal group $\widehat{E}$ of $E$ at $v$ is a Lubin--Tate group of $E$ over $H_v$. However, if $q>7$, this is no longer true because $\psi_{E/H}(v)$ will no longer be a local parameter of $H_v$ in general. We first briefly discuss how one handles this situation. 

Let $v$ be any place of $H$ lying above a prime $w$ of $K$ such that $E$ has good reduction at $v$, and let $\s_v\in G$ be the Frobenius at $v$. Let $\lambda_E(v): E\to E^{\s_v}$ denote the unique isogeny induced by the isogeny $\lambda_{E}(w)$. We remark that the isogeny $\lambda_E(v)$ is defined by the same formulae which define the isogeny $\lambda(v): A(q)\to A(q)^{\sigma_v}$. To see this, recall the notations in the proof of Proposition \ref{thm4} and let $\tau$ be the  nontrivial element of $\Gal(\mc{M}/H)$. Then $E(H)$ is isomorphic to the $-1$ eigenspace for the action of $\Gal(\mc{M}/H)$ on $A(q)(\mc{M})$, that is, the points on $A(q)(\mc{M})$ on which $\t$ acts as $-1$. But we have $\lambda(v)(-P)=-\lambda(v)(P)$ since isogeny preserves the group law, and also we clearly have $\chi_\mc{M}(\tau)=-1$. Hence $\lambda(v)$ is independent of twist by $\chi_\mc{M}$. This induces a homomorphism
\[\widehat{\lambda}_E(v): \widehat{E}\to \widehat{E}^{\s_v},\]
 of formal groups of the curves $E$ and $E^{\s_v}$ at $v$, defined over the ring of integers $\o_v$ of $H_v$. Thus, we can view $\widehat{\lambda}_E(v)$ as an element of $\o_v[[t]]$ satisfying
\begin{equation}\label{LT}\widehat{\lambda}_E(v)(t)\equiv \Lambda(v) t \bmod \text{degree } 2, \;\; \widehat{\lambda}_E(v)(t)\equiv t^{q_w}\bmod v,
\end{equation}
where $\Lambda(v)$ is an element of $\o_v$ and $q_w$ denotes the cardinality of the residue field of the restriction $w$ of $v$ to $K$. Now, we can apply $\s_v^i$ for $i=1,\ldots , f_v$, where $f_v$ denotes the residue degree of $v$ in $H/K$, to $\lambda_E(v)$ and $\widehat{\lambda}_E(v)$. Then we see that 
\[\mr{N}_{H_v/K_w}\Lambda(v)=\psi_{E/H}(v),\]
since $\prod_{i=1}^{f_v} \s_v^i \lambda_E(v)$ is the unique element of $\End_H(E)=\o$ which reduces modulo $v$ to the Frobenius endomorphism at $v$. Thus $\widehat{E}$ is not itself a Lubin--Tate group, but $\widehat{E}$ together with the homomorphism $\widehat{\lambda}_E(v): \widehat{E}\to \widehat{E}^{\s_v}$ is a relative Lubin--Tate group, which was studied by de Shalit in \cite[I \S 1]{dS}. The theory of Lubin--Tate groups generalises to relative Lubin--Tate groups, and in particular, we have the following (\cite[Proposition I.1.8]{dS}):

\begin{thm}\label{thm10} Let $v$ be any place of $H$ where $E$ has good reduction, and let $w$ be its restriction to $K$. Then for any $n\geqs 1$, the extension $H_v(E_{w^n})/H_v$ is totally ramified, and its Galois group is isomorphic to $\left(\o/w^n\right)^\times$.
\end{thm}

We define $F_n=H(E_{\mf{p}^n})$, and $F=F_2$ or $F_1$, according as $p=2$ or $p>2$. Set
\[F_\i=H(E_{\p^\i}), \;\;\; \mf{H}=\Gal(F_\i/H).\] 
Then by Theorem \ref{thm10}, we have a character $\chi_\p: \mf{H}\to \o_\p^\times=\Z_p^\times$ giving the action of $\mf{H}$ on $E_{\p^\i}$, which is an isomorphism. We write
 $\mf{H}=\D\times \G$, where $\D=\Gal\left(F/H\right)$ is cyclic of order $2$ or $p-1$ if $p=2$ or $p>2$ and $\G=\Gal\left(F_\i/F\right)$ is isomorphic to $\Z_p$.

\section{Construction of the $\mf{p}$-adic $L$-functions}\label{ch4}
\subsection{Construction of the $\mf{p}$-adic $L$-function for $H_\infty/H$}\label{section4.2}~
\vspace{5 pt}\\
Let $K_\infty$ denote the unique $\mathbb{Z}_p$-extension of $K$ unramified outside $\p$, and let $H_\infty$ denote the composite field $HK_\infty$. Then $H_\infty$ is a subfield of $F_\infty$ such that $H_\infty/H$ is a $\mathbb{Z}_p$-extension, and it is clear that $H_\infty=F_\infty^\Delta$. We now construct the $\mf{p}$-adic $L$-function which interpolates the values of the $L(\ovl{\psi}_{E/H}^k,k)$ for integers $k\geqslant 1$ with $k\equiv 0 \bmod \#(\Delta)$. This gives rise to the $\mf{p}$-adic $L$-function for $H_\infty/H$, and it is an essential ingredient for the main conjecture related to the Birch--Swinnerton-Dyer conjecture for $E/H$ at $p=2$. The construction of the $\mf{p}$-adic $L$-function interpolating the values at $k\equiv 1\bmod \#(\D)$ can be found in \cite{kez}. We will follow the ideas in \cite{coa-gol} which deals with the case $p>2$. The case $p=2$ cannot be found in literature.

Write $x$, $y$ for the coordinates of $E/H$. We fix a generalised global minimal Weierstrass equation for $E$ over $H$, which exists by \cite[Proposition 3.2]{gr2}, to be
\begin{equation}\label{eq4.1}y^2+a_1 x y+a_3 y=x^3+a_2 x^2+a_4 x+a_6.
\end{equation}
Recall that $G$ denotes the Galois group of $H$ over $K$. Then applying $\sigma\in G$ to \eqref{eq4.1} gives a generalised global minimal Weierstrass equation for $E^\sigma/H$.
Let $\omega^\sigma$ be the N\'{e}ron differential on $E^\sigma$, and note that the discriminant of this equation $\Delta(E^\sigma)$ is equal to $(\Delta(E))^\sigma=\D(E)$. 
Let $L$ (resp. $L_\sigma$) be the period lattice of the Neron differential on our global minimal Weierstrass equation for $E$ (resp. $E^\sigma$). Then there exists $\O_\i\in\C^\times$ such that $L=\O_\i\o$. 
 The uniformisation $\Phi: \C/L\xto{\sim} E(\C)$ is accomplished through
\begin{small}
\[\Phi(z,L)=\left(\wp(z,L)-\frac{((a_1)^2+4a_2)}{12}, \frac{1}{2}\left(\wp'(z,L)-a_1\left(\wp(z,L)-\frac{((a_1)^2+4a_2)}{12}\right)-a_3\right)\right).\]
\end{small}

Given a principal ideal $\mf{a}=(\alpha)$ with $\alpha\in\o$ and $(\mf{a},6\mf{f})=1$, define
\[R_\mf{a}(P)=c_E(\mf{a})\prod\limits_U\left(x(P)-x(U)\right)^{-1},\]
where $U$ runs over any set of representatives of $E_\mf{a}\backslash\{\o\}$ modulo $\{\pm 1\}$, and $c_E(\mf{a})$ is an element of $H$ whose $12$th power is equal to $\D(E)^{\mathrm{N}\mf{a}-1}/\Lambda(\mf{a})^{12}$, where $\Lambda(\mf{a})\in H^\times$ satisfies
\[\lambda_E(\mf{a})^*(\omega^{\sigma_\mf{a}})=\Lambda(\mf{a})\omega.\]
  Thus $R_\mf{a}(P)$ is a rational function on $E$ with coefficients in $H$. Let us write $P$ for the generic point on $E^\sigma$ with coordinates $(x, y)$. Applying $\sigma\in G$ to the coefficients of $R_\mf{a}(P)$, we obtain a rational function $R_\mf{a}^\sigma(P)$ on the curve $E^\sigma/H$. For ease of notation, we will work with $E$ but the arguments are identical if we replace this by $E^\sigma$ and consider rational functions on $E^\sigma$ over $H$. 
  
\begin{prop}\label{prop3}  Let $\mf{b}$ be an integral ideal of $K$ with $(\mf{b}, \mf{a})=1$. Then we have
\[R_\mf{a}^{\sigma_\mf{b}}(\lambda_{E}(\mf{b})(P))=\prod\limits_{R\in E_\mf{b}}R_\mf{a}(P \oplus R).\]
\end{prop}

\begin{proof} Recall that the kernel of $\lambda_{E}(\mf{b})$ is $E_\mf{b}$, and $\lambda_{E}(\mf{b})$  is injective on $E_\mf{a}$ since $(\mf{b},\mf{a})=1$. Hence, the left hand side and the right hand side of the above equation have the same divisor, and 
\[\frac{R_\mf{a}^{\sigma_\mf{b}}(\lambda_{E}(\mf{b})(P))}{\prod\limits_{R\in E_\mf{b}}R_\mf{a}(P \oplus R)}\]
is a non-zero element of $H$. It can be shown, thanks to the unique scaling factor $c_E(\mf{a})$ in our definition of the rational functions, that this constant is equal to $1$. See \cite[Appendix, Theorem 4]{coa} for details.
\end{proof}

Let $k$ be a positive even integer, so that the conductor of $\varphi_K^k$ is $(1)$. We write $P_n$ for a primitive $\p^{n}$-division point of $E$. Note that $R_\mathfrak{a}(P)$ has a zero of order $\mathrm{N}\mathfrak{a}-1$ at $P=\mathcal{O}$, and $R_\mf{a}(P_n)$ is not a unit. To get rid of this zero at $P=\o$, define the index set\\
\begin{small}
 \begin{equation}\label{defi}I=\{(\mf{a}_i,n_i), \;\; i=1,\ldots , r, \;\mf{a}_i=(\alpha_i)\sb \mathcal{O},(\mf{a}_i,6\mf{p})=1, n_i\in\Z \text{ with }\sum\limits_{i=1}^rn_i(\mathrm{N}\mf{a}_i-1)=0\}.
\end{equation}
\end{small}
Given $\mc{D}=(\mf{a}_i,n_i)\in I$, define
\[R_\mc{D}(P)=\prod_{i=1}^rR_{\mf{a}_i}(P)^{n_i}.\]
Then $R_\mc{D}(P)$ has no zero at $P=\o$, and $R_\mc{D}(P_n)$ is a unit, as we will see in Corollary \ref{cor2}. 

Define $G_k(L)=\sum\limits_{w\in L\backslash\{0\}}\frac{1}{w^k}$ for $k\geqs 3$, $G_2(L)=\lim_{s\to 0+}\sum\limits_{w\in L\backslash\{0\}}w^{-2}|w|^{-2s}$ and $G_1(L)=0$.

\begin{prop} Let $\mf{s}$ be an integral ideal of $K$ prime to $\mf{f}$ such that $\sigma_\mf{s}=\sigma$. Then for any $\mc{D}=(\mf{a}_i,n_i)\in I$ and $k\geqs 2$ an even integer, we have
\[\left(\frac{d}{dz}\right)^k\log R_\mc{D}^\sigma(\Phi(z,L_\sigma))|_{z=0}=\sum\limits_{i=1}^r-n_i(k-1)!\frac{\varphi_K^k(\mf{s})}{\Lambda(\mf{s})^k\Omega_\infty^k}\left(\mathrm{N}\mf{a}_i -\alpha_i^k\right)L(\ovl{\varphi}_K^k, \sigma,k)).\]
\end{prop}

\begin{proof}Let $\mf{L}=\mathbb{Z}\omega_1+\mathbb{Z}\omega_2$ be a complex lattice, whose basis is ordered so that $\omega_1/\omega_2$ belongs to the upper half plane. We will modify the Weierstrass $\s$-function slightly, and define
\[\Theta(z,\mf{L})=\exp\left\{-G_2(\mf{L})\frac{z^2}{2}\right\}\s(z,\mf{L}).\]
Recall that for any integer $k\geqs 1$, we can define the Kronecker--Eisenstein series
\[H_k(z,s,\mf{L})=\sum\limits_{w\in \mf{L}}\frac{(\overline{z}+\overline{w})^k}{\:\: |z+w|^{2s}},\]
where the sum in taken over all $w\in \mf{L}$, except $-z$ if $z\in\mf{L}$. This series converges for $\rm{Re}(s)>\frac{k}{2}+1$, and it has analytic continuation to the whole complex $s$-plane. The non-holomorphic Eisenstein series $\mathcal{E}_k^*(z,\mf{L})$ is defined by $\mathcal{E}_k^*(z,L):=H_k(z,k,\mf{L})$.
Furthermore, it is well-known that (see \cite[Corollary 1.7]{gol-sch}) for any $z_0\in \C\backslash \mf{L}$, we have

\begin{equation}\label{eq4.1.7}\frac{d}{dz}\log \Theta(z+z_0,\mf{L})=\ovl{z}_0A(\mf{L})^{-1}+\sum\limits_{k=1}^\infty (-1)^{k-1}\mathcal{E}_k^*(z_0,\mf{L})z^{k-1},
\end{equation}
where $A(\mf{L})=\frac{\omega_1\ovl{\omega_2}-\omega_2\ovl{\omega_1}}{2\pi i}$. By \cite[Theorem 1.9]{gol-sch}, for any principal integral ideal $\mf{a}=(\alpha)$ with $(\mf{a},6\mf{f})=1$, we have
\[\frac{\Theta^2(z,L_\sigma)^{\mathrm{N}\mf{a}}}{\Theta^2(z,\alpha^{-1}L_\sigma)}=\prod\limits_{\substack{w\in\alpha^{-1}L_\sigma/L_\sigma\\ w\neq 0}}(\wp(z,L_\sigma)-\wp(w,L_\sigma))^{-1},\]
so we can write
\[R_\mc{D}^\sigma(\Phi(z,L_\sigma))^2=\prod\limits_{i=1}^r \left(c_E(\mf{a_i})^2\frac{\Theta^2(z,L_\sigma)^{\mathrm{N}\mf{a_i}}}{\Theta^2(z,\alpha_i^{-1}L_\sigma)}\right)^{n_i}.\]
Now, \eqref{eq4.1.7} gives
\[\frac{d}{dz}\log \Theta(z,L_\sigma)=\sum\limits_{k=1}^\i(-1)^{k-1}G_k(L_\sigma)z^{k-1}\]
and $G_k(L_\sigma)=0$ for $k$ odd. Therefore,
\begin{align*}\frac{d}{dz}\log R_\mc{D}^\sigma(\Phi(z,L_\sigma))=\sum\limits_{i=1}^r \sum\limits_{\substack{k\geqs 2\\ k\text{ even}}}-n_i(\mathrm{N} \mf{a}_i -\alpha_i^k)G_k(L_\sigma)z^{k-1}
\end{align*}
by the homogeneity of $G_k$. Let $\mf{b}$ be an ideal of $K$. Setting $k=s$ , $\mf{g}=(1)$ and $\rho=\Omega_\infty$ in \cite[Proposition 5.5]{gol-sch}, we see that the partial Hecke $L$-function $L(\ovl{\varphi}_K^k, \sigma_{\mf{b}},k)$ is identically equal to
\[G_k(L_{\sigma_\mf{b}})=\frac{\varphi_K^k(\mf{b})}{\Lambda(\mf{b})^k\Omega_\infty^k}L(\ovl{\varphi}_K^k, \sigma_{\mf{b}},k).\]
Hence, setting $\mf{b}=\mf{s}$ and noting $\varphi_K^k(\mf{a}_i)=\alpha_i^k$ ($k$ is even), we obtain
\begin{align*}(\mathrm{N} \mathfrak{a}_i -\alpha_i^k)G_k(L_{\sigma})=\frac{\varphi_K^k(\mf{s})}{\Lambda(\mf{s})^k\Omega_\infty^k}\left(\mathrm{N} \mf{a}_i -\alpha_i^k\right) L(\ovl{\varphi}_K^k, \sigma,k).
\end{align*}
This completes the proof of the proposition.
\end{proof}

Define
\[\Psi_\mc{D}(P)=\frac{R_\mc{D}(P)^{\mr{N}\p}}{R_\mc{D}^{\sigma_\mf{p}}(\lambda_{E}(\p)(P))},\]
so that we have
\begin{equation}\label{eq1}
\prod\limits_{R\in E_\mf{p}}\Psi_\mc{D}(P\oplus R)=1.
\end{equation}

Now, we fix an embedding $i_v: \ovl{K}\to \ovl{K}_\mf{p}$, and we let $v$ denote the prime of $H$ above $\mf{p}$ determined by $i_v$. Let $t=-\frac{x}{y}$ be a parameter for this formal group. Let $A_\mc{D}(t)$ be the development as a power series in $t$ of the rational function  $\Psi_\mc{D}(P)$. We will show that $A_\mc{D}(t)\in 1+\mf{P}\o_\mf{P}[[t]]$, and so $C_\mc{D}(t)=\frac{1}{\mr{N}\p}\log A_\mc{D}(t)\in \o_\mf{P}[[t]]$.

\begin{lem} Let $B_\mc{D}(t)$ denote the $t$-expansions of $R_\mc{D}(P)$. Then $B_\mc{D}(t)$  is a unit in $\o_v[[t]]$.
\end{lem}

\begin{proof} Let $U$ be any non-zero element of $E_{\mf{a}_i}$. Let $\mf{P}$ denote any prime of $H(U)$ above $v$, and let $\o_\mf{P}$ be the ring of integers of the completion of $H(U)$ at $\mf{P}$. Note that $x(U)$ is integral at $\mf{P}$ because $(\mf{a}_i,p)=1$. Thus we see that the coefficients of the $t$-series expansion of $x(P)-x(U)$ all belong to $\o_\mf{P}$. Moreover, $R_\mc{D}(P)$ is holomorphic at $t=0$, and so there are no negative powers of $t$ in its $t$-series expansion. It remains to show that the constant term of the $t$-series expansion of $R_\mc{D}(P)$ is a unit. This follows from Corollary \ref{cor2}, where we show that $R_\mc{D}(P_n)$ is a global unit for a point $P_n$ of $E$ with $\ord_\mf{P}(P_n)<0$.
\end{proof}

From this we obtain
\begin{cor}  Let $A_\mc{D}(t)$denote the $t$-expansion of $\Psi_\mc{D}(P)$. Then $A_\mc{D}(t)$ belongs to $1+\mf{m}_v[[t]]$, where $\mf{m}_v$ denotes the maximal ideal of $\o_v$.
\end{cor}

\begin{proof}
Write $B_\mc{D}(t)=\sum\limits_{n=0}^\i a_nt$. Thus, by the previous lemma, $a_n\in\o_v$ for all $n\geqs 0$ and $a_0\in\o_v^\times$. Now, $A_\mc{D}(t)=\frac{B_\mc{D}(t)^p}{B_\mc{D}^{\sigma_\mathfrak{p}}\left(\widehat{\lambda}_E(v)(t)\right)}$, where $\widehat{\lambda}_E(v): \widehat{E}\to \widehat{E}^{\s_\mathfrak{p}}$ is a homomorphism of formal groups of the curves $E$ and $E^{\s_\mathfrak{p}}$ at $v$ induced by $\lambda_E$, defined over $\o_v$ of $H_v$. Then \eqref{LT} gives us $\widehat{\lambda}_E(v)(t)\equiv t^p\bmod v$, and we see that
\[B_\mc{D}^{\sigma_\mathfrak{p}}\left(\widehat{\lambda}_E(v)(t)\right)=\sum\limits_{n=0}^\i a_n^{\sigma_v}(\widehat{\lambda}_E(v)(t))^n\equiv \sum\limits_{n=0}^\i a_n^p t^{pn}\bmod v.\]
On the other hand,
\[B_\mc{D}(t)^p \equiv \sum\limits_{n=0}^\i (a_nt^{n})^p \equiv \sum_{n=0}^\i a_n^p t^{pn}\bmod v \]
so $A_\mc{D}(t)\equiv 1\bmod v$, as required.
\end{proof}

\begin{lem} \label{lem4.1.4} Define $C_\mc{D}(t)$ by 
\[C_\mc{D}(t)=\frac{1}{p}\log A_\mc{D}(t).\]
Then $C_\mc{D}(t)\in\o_v[[t]]$, and 
\begin{equation}\label{eqC2}\sum\limits_{\w\in \mathcal{D}_{\sigma, p}}C_\mc{D}(t[+]\w)=0,
\end{equation}
where $\mathcal{D}_{\sigma, p}$ denotes the group of $p$-division points on the formal group $\widehat{E}$ at a place $v$ of $H$ lying above $\mf{p}$ and $[+]$ denotes the group law on $\widehat{E}$. This group can be identified with $E_\mf{p}$.
\end{lem}

\begin{proof} We have

\[C_\mc{D}(t)=\frac{1}{p}\sum\limits_{n\geqs 1}\frac{(-1)^{n-1}(A_\mc{D}(t)-1)^n}{n}.\]
The first claim is now clear from the previous lemma as $n\geqs \ord_vb(n)+1$ for $n\geqs 1$. The final equation then follows from \eqref{eq1}.
\end{proof}

Let $\ms{I}$ be the ring of integers of the completion of the maximal unramified extension $K_\mf{p}^\text{ur}$ of $K_\mf{p}$. By \cite[Proposition 1.6]{dS}, we have an isomorphism
\[\d_{\sigma, v}: \widehat{\mb{G}}_m\xto{\sim} \widehat{E}^\sigma\]
defined over $\mathscr{I}$, where $\widehat{\mb{G}}_m$ denotes the formal multiplicative group and $\widehat{E}^\sigma$ denotes the formal group of $E^\sigma$ at $v$. Let $J_\mc{D}^\sigma(W)=C_\mc{D}\circ \d_{\sigma, v}(W)\in \mathscr{I}[[W]]$, and let $\mu_{\mc{D},\sigma}$ be the $\mathscr{I}$-valued measure on $\Z_p$ determined by $J_\mc{D}^\sigma(W)$, that is,
\begin{equation}\label{eq3}J_\mc{D}^\sigma(W)=\int_{\Z_p}(1+W)^x d\m_{\mc{D},\sigma}(x).
\end{equation}
We claim that the measure $\m_{\mc{D},\sigma}$ is supported on $\Z_p^\times$. Indeed, let $\Lambda_\ms{I}(\Z_p)$ (resp. $\Lambda_\ms{I}(\Z_p^\times)$) be the ring of $\ms{I}$-valued measures on $\Z_p$ (resp. $\Z_p^\times$). Then we have an inclusion $\iota: \Lambda_\ms{I}(\Z_p^\times)\inj \Lambda_\ms{I}(\Z_p)$ given by extending the measures on $\Z_p$ to $\Z_p^\times$ by zero. Given a measure $\mu$ in $\Lambda_\ms{I}(\Z_p)$, write $f_\mu(W)\in \ms{I}[[W]]$ for the corresponding power series given by the isomorphism $\Lambda_\ms{I}(\Z_p)\isom \ms{I}[[W]]$. Then it is well-known (see \cite[I.3.3]{dS} for more details) that $\mu$ belongs to $\iota\left(\Lambda_\ms{I}(\Z_p^\times)\right)$ if and only if $f_\mu$ satisfies the equation
\[\sum_{\zeta\in \bold{\mu}_p}f_\mu(\zeta(1+W)-1)=0.\]
It follows from \eqref{eqC2} that this is satisfied by $J_\mc{D}^\sigma$. Writing also $\mu_{\mc{D},\sigma}$ for the corresponding measure in $\Lambda_\ms{I}(\mf{H})$, we have
\begin{equation}\label{eq7}\int_\mf{H}\chi_\p^kd\m_{\mc{D},\sigma}=\int_{\Z_p}x^kd\m_{\mc{D},\sigma}=D^kJ_\mc{D}^\sigma(W)|_{W=0},
\end{equation}
where $D=(1+W)\frac{d}{dW}$. We have an isomorphism $\widehat{\mb{G}}_m\xto{\sim}\widehat{\mb{G}}_a$ given by $W\mto e^z-1$, hence we see immediately that $D=\frac{d}{dz}$.
Moreover, we have $\d_{\sigma, v}(W)=\O_{\sigma,v} W+\cdots$ with $\O_{\sigma,v}\in \mathscr{I}^\times$, so
\begin{equation}\label{eq8}D^k J_\mc{D}^\sigma(W)|_{W=0}=\left(\frac{d}{dz}\right)^k(J_\mc{D}^\sigma(e^z-1))|_{z=0}=\frac{1}{\mr{N}\p}\Omega_{\sigma,v}^k \left(\frac{d}{dz}\right)^k\log \Psi_\mc{D}^\sigma(\Phi(z,L_\sigma))|_{z=0}.
\end{equation}

\begin{lem}\label{lem5} We have $\Omega_{\sigma,v}=\Lambda(\mf{s})\Omega_v$, where $\Omega_v\in \ms{I}^\times$ is the coefficient of $W$ in the formal power series $t = \delta_v(W)$, and $\delta_v: \widehat{\mb{G}}_m\xto{\sim} \widehat{E}$ is an isomorphism defined over $\mathscr{I}$.
\end{lem}
\begin{proof}We have $\lambda_{E}(\mf{s})^*(\omega^\sigma)=\Lambda(\mf{s})\omega$ by definition, so that $\lambda_E(\mf{s})\left(\Phi(z,L)\right)=\Phi(\Lambda(\mf{s})z, L_\sigma)$. Hence, writing $\exp(z,L_\sigma)$ for the formal power series in $z$ obtained by expressing $t=-x/y$ in terms of $z$ using the isomorphism $\Phi(z,L_\sigma)$ for $E^\sigma$, we also have $\lambda_E(\exp(z,L))=\exp(\Lambda(\mf{s})z,L_\sigma))$. Now, regarding $z$ as the parameter of the formal additive group, $\exp(z,L_\sigma)$ is the exponential map of $\widehat{E}^\sigma$. It then follows by the uniqueness of the exponential maps for the formal groups that
\[\delta_{\sigma,v}(e^{z/\Omega_{\sigma,v}}-1)=\exp(z,L_\sigma).\]
On the other hand, we have $\delta_{\sigma,v}=\widehat{\lambda}_E(\mf{s})\circ \delta_v(W)$, where $\widehat{\lambda}_E(\mf{s}):\widehat{E}\to \widehat{E}^\sigma$ is the isomorphism over $H_v$ of formal groups induced by $\lambda_E(\mf{s})$. Hence we have
\[\delta_{\sigma,v}(e^z-1)=\exp(\Lambda(\mf{s})\Omega_v z, L_\sigma).\]
The assertion follows by comparing the coefficients of $z$ in the above equations.
\end{proof}

\begin{lem}\label{lem14} For an even integer $k\geqs 2$, we have
\begin{small}
\begin{align*}\Lambda(\mf{s})^{-k}\Omega_v&^{-k}\int_\mf{H}\chi_\p^k d \mu_{\mc{D},\sigma}=\\
&\sum_{i=1}^r -n_i (k-1)!\varphi_K^k(\mf{s})\Lambda(\mf{s})^{-k}\Omega_\infty^{-k}c_k(\mf{a}_i)\left(L(\ovl{\varphi}_K^k,\sigma,k)-\frac{\varphi_K^k(\p)}{\mathrm{N}\p}L(\ovl{\varphi}_K^k,\sigma\sigma_\p, k)\right),
\end{align*}
\end{small}
where $c_k(\mf{a}_i)=\mathrm{N}\mf{a}_i-\alpha_i^k$.
\end{lem}
\begin{proof}
We have $\lambda_E(\p)\Phi(z,L_\sigma)=\Phi(\Lambda(\p)^\sigma z, L_{\sigma\sigma_\p})$ and $\Lambda(\mf{sp})=\Lambda(\mf{s})\Lambda(\p)^\sigma$, so
\begin{align*}\left(\frac{d}{dz}\right)^k\log R_\mc{D}^{\sigma\sigma_\mf{p}}(\lambda_E(\p)\Phi(z,L_\sigma))|_{z=0}=\sum\limits_{i=1}^r-n_i(k-1)!\frac{\varphi_K^k(\mf{sp})}{\Lambda(\mf{s})^k\Omega_\infty^k}c_k(\mf{a}_i)L(\ovl{\varphi}_K^k, \sigma\sigma_\p,k).
\end{align*}
Therefore,
\begin{small}
\begin{align}\label{eq9}\left(\frac{d}{dz}\right)^k\log \Psi_\mc{D}^\sigma(\Phi&(z,L_\sigma))|_{z=0}=\\
&\sum_{i=1}^r -n_i (k-1)!\frac{\varphi_K^k(\mf{s})\mr{N}\p}{\Lambda(\mf{s})^k\Omega_\infty^k}c_k(\mf{a}_i)\left(L(\ovl{\varphi}_K^k,\sigma,k)-\frac{\varphi_K^k(\p)}{\mathrm{N}\p}L(\ovl{\varphi}_K^k,\sigma\sigma_\p, k)\right). \nonumber
\end{align}
\end{small}
Combining \eqref{eq7}, \eqref{eq8} and \eqref{eq9}, the proof of the proposition is complete.
\end{proof}

Define 
\[ G^*=\Hom(G,\mb{C}_p^\times),\]
where $G=\Gal(H/K)$ as before. For each $\chi\in G^*$, put
\[D_{\mc{D}}(\chi,k)=\sum_{\sigma\in G} \chi(\sigma)\varphi_K(\mf{s})^{-k}\int_\mf{H}\chi_\p^k d \mu_{\mc{D},\sigma}.\]
We conclude immediately that
\[D_{\mc{D}}(\chi,k)=c_k(\mc{D})(k-1)!\left(\sum_{\sigma\in G}\chi(\sigma)L(\ovl{\varphi}_K^k,\sigma,k)\frac{\Lambda(\mf{s})^k\Omega_v^{k}}{\Lambda(\mf{s})^{k}\Omega_\infty^{k}}\right)\left(1-\frac{\varphi_K^k(\p)\chi^{-1}(\sigma_\p)}{\mr{N}\p}\right),\]
where $c_k(\mc{D})=\sum_{i=1}^r -n_i c_k(\mf{a}_i)$. Let $\mc{C}$ denote a set of integral ideals representing of the ideal class group of $K$ with $(\mf{c},\mf{pf})=1$ for any $\mf{c}\in \mc{C}$, and set $\Omega_\infty(E/H)=\prod_{\mf{c}\in \mc{C}} \Lambda(\mf{c})\Omega_\infty$ and $\Omega_\mf{p}(E/H)=\prod_{\mf{c}\in \mc{C}} \Lambda(\mf{c})\Omega_v$. Recalling 
\[L(\ovl{\psi}_{E/H}^k,k)=\prod_{\chi\in G^*}\sum_{\sigma\in G}\chi(\sigma)L(\ovl{\varphi}_K^k,\sigma,k)\]
and the factorisation of primes of $K$ in $H$ given by class field theory, we obtain the following lemmas:

\begin{lem} For any even integer $k\geqs 2$, we have
\begin{footnotesize}
\begin{align*}\prod_{\chi\in G^*}D_{\mc{D}}(\chi,k)=c_k(\mc{D})^h\left((k-1)!\right)^h\Omega_\mf{p}(E/H)^{k}\Omega_\infty(E/H)^{-k}L(\ovl{\psi}_{E/H}^k,k)\cdot \prod_{w\mid \p}\left(1-\frac{\psi_{E/H}^k(w)}{\mr{N}w}\right).
\end{align*}
\end{footnotesize}
\end{lem}

\begin{lem}\label{lem4.2} There exists a measure $\nu_\mc{D}$ in $\Lambda_\ms{I}(\mf{G})$ such that for all $k\geqs 1$, $k\equiv 0\bmod \#(\Delta)$, we have
\begin{small}
\[\Omega_\mf{p}(E/H)^{-k}\int_{\mf{G}}\chi_\p^k d\nu_\mc{D}=c_k(\mc{D})^h\left((k-1)!\right)^h\Omega_\infty(E/H)^{-k}L(\ovl{\psi}_{E/H}^k,k)\cdot \prod_{w\mid \p}\left(1-\frac{\psi_{E/H}^k(w)}{\mr{N}w}\right).\]
\end{small}
\end{lem}

Note that, since $k\equiv 0\bmod \#(\Delta)$, we have $\chi_\mf{p}^k(\tau)=1$ for any $\tau \in \Delta$. Hence, we can naturally consider $\nu_\mc{D}$ as an element of $\Lambda_\ms{I}(\ms{G})$. The only dependence of $\nu_\mc{D}$ on $\mc{D}$ occurs in the factor $c_k(\mc{D})^h$. We claim that we can remove this factor and obtain a pseudo-measure which is independent of $\mc{D}$.

\begin{lem}\label{lem2}There exists an element $\mc{D}$ in the index set $I$ defined in \eqref{defi} and $\theta_\mc{D}\in \Lambda_\ms{I}(\ms{G})$ such that $\theta_{\mc{D}}|_\Gamma$ generates the augmentation ideal of $\Lambda_\ms{I}(\G)\sb \Lambda_\ms{I}(\ms{G})$ and
\[\int_{\ms{G}}\chi_\p^kd\theta_\mc{D}=c_k(\mc{D})^h\]
for all $k\geqs 1$.
\end{lem}
\begin{proof}
Choose $\a\in\o$ so that $\a\equiv1\bmod \p^{m+1}$, $\a\equiv 1+p^m \bmod \p^{*m+1}$ where $m=1$ or $2$ according as $p>2$ or $p=2$, and define $\mf{a}=(\a)$. Take $\mf{a}_1=\mf{a}$, $\mf{a}_2=\ovl{\mf{a}}$, $n_1=1$, $n_2=-1$. Then $(\{\mf{a}_1,\mf{a}_2\},\{n_1,n_2\})\in I$. Write $\sigma_\mf{a}$ for the Artin symbol $(\mf{a}, H_\infty/K)$ of $\mf{a}$ for $H_\infty/K$. Note that $(\mf{a},H/K)=1$ since $\mf{a}$ is principal, so that we can consider $\sigma_\mf{a}$ as an element of $\Gamma$. We will show that the measure
\begin{align*}\theta_\mc{D}=-(\mathrm{N}\mf{a}-\s_\mf{a}-(\mathrm{N}\ovl{\mf{a}}-\s_{\ovl{\mf{a}}}))=\s_\mf{a}-\s_{\ovl{\mf{a}}},
\end{align*}
has the desired property. Indeed, we have $\chi_\p^k(\theta_\mc{D})=c_k(\mc{D})^h$, so it remains to show that $\theta_{\mc{D}}|_\G$ generates the augmentation ideal of $\mathscr{I}[[\G]]$. In order to do this, let us fix a topological generator of $\c$ of $\G$, and write $\s_\mf{a}|_\G=\c^a$, $\s_{\ovl{\mf{a}}}|_\G=\c^b$ where $a,b\in\Z_p$. It suffices to show that $\theta_{\mc{D}}|_\G=(1-\c)\cdot u$ for $u\in\Z_p[[\G]]^\times$. Now, we have $\G\simeq \Z_p$ and $\frac{1}{p^m}\log: 1+p^m\Z_p\to \Z_p$ sending $1+p^mx\mto \frac{1}{p^m}\sum\limits_{i=1}^\i (-1)^{i-1}\left(\frac{p^m x}{i}\right)^i$ is an isomorphism.  Hence $\p^{m+1}\mid \a-1$ implies $a\equiv 0\bmod p$, and $\a^*$ generates $1+p^m\o_\p$ so $b\not\equiv 0\bmod p$. Now,
\begin{align*}\s_\mf{a}|_\G-\s_{\ovl{\mf{a}}}|_\G=\c^a-\c^b=\c^a(1-\c^{b-a}),
\end{align*}
where clearly $\c^a$ is a unit, and also  $b-a\not\equiv 0\bmod p$ so $1-\c^{b-a}$ is a product of $(1-\c)$ and a unit, as required.
\end{proof}

We define
\begin{equation}\label{eqnup}\nu_\mf{p}=\nu_\mc{D}/\theta_\mc{D}.
\end{equation}

This is a pseudo-measure, since $(1-\c)\cdot\frac{1}{\theta_\mc{D}}$ is a unit by the proof of Lemma \ref{lem2}. The following is an immediate consequence of Lemma \ref{lem4.2} and Lemma \ref{lem2}.

\begin{thm}\label{thm4.2.7} There exists a unique element $\nu_\mf{p}$ belonging to the quotient field $\Lambda_\ms{I}(\ms{G})$ such that, for all integers $k\geqs 1$ with $k\equiv 0\bmod \#(\D)$, we have
\[\O_\mf{p}(E/H)^{-k}\int_{\ms{G}}\chi_\mf{p}^k d\nu_\mf{p}=\left((k-1)!\right)^h \O_\i(E/H)^{-k} L(\ovl{\psi}_{E/H}^k,k) \prod_{v\in P}\left(1-\frac{\psi_{E/H}^k(v)}{\mr{N}v}\right).\]
Furthermore, the denominator of $\nu_\mf{p}$ is given by $\gamma-1$, so that $(\gamma-1)\nu_\mf{p}\in \Lambda_\ms{I}(\ms{G})$.
\end{thm}

If we in addition assume $(p,h)=1$, the idempotents $e_\chi:=\frac{1}{\#(G)}\sum_{g\in G}\chi^{-1}(g)g$ corresponding to any $\chi\in G^*$ lie inside $\Lambda_\ms{I}(\ms{G})$, and thus we can decompose $\nu_\mf{p}$ as a sum of elements in $e_\chi\Lambda_\ms{I}(\Gamma)$. Given $\chi\in G^*$, let $\nu_\mf{p}^\chi\in\Lambda_\ms{I}(\Gamma)$ denote the $\chi$-part of $\nu_\mf{p}$ in the decomposition. Then we have shown that $\nu_\mf{p}^\chi\in \Lambda_\ms{I}(\Gamma)$ for every $\chi\neq 1$, and $(\gamma-1)\nu_\mf{p}^\chi \in \Lambda_\ms{I}(\Gamma)$ for $\chi=1$. Thus, identifying $\Lambda_\ms{I}(\Gamma)$ with $\ms{I}[[T]]$ via the map sending $\gamma$ to $1+T$, we have $\nu_\mf{p}^\chi\in \ms{I}[[T]]/T$ when $\chi$ is trivial. The pseudo-measure $\nu_\mf{p}$ will be used for the main conjecture for $H_\infty/H$.

\subsection{Elliptic Units}\label{section4.3}~
\vspace{5 pt}\\
In this section, we will use the rational function $R_\mc{D}$ to generate elliptic units.

For $n\geqs 1$ and $\sigma\in G$, let $P_n^{(\sigma)}$ be a primitive $\mf{p}^n$-division point on $E^\sigma$ satisfying $\lambda_{E^\sigma}(\mf{p}) P_n^{(\sigma)}=P_{n-1}^{(\sigma\sigma_\mf{p})}$, where $\sigma_\mf{p}$ is the Artin symbol of $\mf{p}$ for $H/K$. Note that we can write $P_n^{(\sigma)}=\Phi(\rho, L_\sigma)$ for some $\rho\in\mathfrak{p}^{-n}L_\sigma\backslash L_\sigma$. Given an  integral ideal $\mf{b}$ of $K$ prime to $\mf{a}_i$ and $\mf{p}$, the image of $P_n$ under the Artin symbol of $\mf{b}$ for $H(E_{\mf{p}^n})/K$ is $\lambda_{E}(\mf{b})(P_n)$, so a choice of  $P_n^{(\sigma_\mf{b})}$ for the Artin symbol $\sigma_\mf{b}$ of $\mf{b}$ for $H/K$ is given by 
\[P_n^{(\sigma_\mf{b})}=\Phi(\Lambda(\mf{b})\rho, L_{\sigma_\mf{b}}),\]
which is a point on $E^{\sigma_\mf{b}}$.

It can be shown that $R_\mf{a}(P_n)\in K(\p^n)$, where $K(\p^n)$ denotes the ray class field of $K$ modulo $\p^n$ (see \cite[Theorem 4.3.1]{kez}). 

\begin{prop}\label{cor1}For any integer $m\geqs 2$, we have
\[\mr{N}_{F_m/F_{m-1}}R_\mf{a}(P_m)=R_\mf{a}^{\sigma_\mf{p}}(P_{m-1}^{(\sigma_\mf{p})}),\]
where $\sigma_\mf{p}=(\mf{p}, H/K)$ denotes the Artin symbol of $\mf{p}$ for the extension $H/K$.
\end{prop}

\begin{proof} Write $\Phi(v,L)=P_m$. The conjugates $\Phi(v,L)^\tau$ of $\Phi(v,L)$ as $\tau$ runs over $\Gal(F_{mh}/F_{(m-1)h})$ are $\Phi(v+u, L)$ for $\Phi(u,L)\in E_\mf{p}$. Hence
\[\mathrm{N}_{m,n}R_\mf{a}(\Phi(v,L))=\prod_{u\in \mf{p}^{-1}L/L}R_\mf{a}(\Phi(v+u, L)).\]
But by Proposition \ref{prop3}, the right hand side is equal to $R_\mf{a}^{\sigma_\mf{p}}\left(\lambda_{E}(\mf{p})(\Phi(v, L))\right)=R_\mf{a}^{\sigma_\mf{p}}\left(\Phi(\Lambda(\mf{p}) v, L_{\sigma_\mf{p}})\right)$, and $\Phi(\Lambda(\mf{p}) v, L_{\sigma_\mf{p}})$ is a primitive $\mf{p}^{m-1}$ torsion point of $E^{\sigma_\mf{p}}$. Hence $\Phi(\Lambda(\mf{p}) v, L_{\sigma_\mf{p}})=P_{m-1}^{(\sigma_\mf{p})}$ by our choice of $\mf{p}$-power torsion points.
\end{proof}

Let  $L$ be an arbitrary finite extension of $K$. We say that $a\in L$ is a \emph{universal norm} from $L(E_{\mf{p}^\infty})$ if it is a norm from $L(E_{\mf{p}^n})$ for every $n\geqs 0$. The following is well-known (see \cite[Lemma 5]{coa}).

\begin{lem}\label{lem15} Let $L$ be a finite extension of $K$, and $a\in L^\times$ a universal norm from $L(E_{\mf{p}^\infty})$. Then every prime which divides $a$ lies above $\mf{p}$.
\end{lem}

\begin{cor}\label{cor2} $R_\mc{D}(P_n)$ are global units.
\end{cor}

\begin{proof} We note that if $\mc{D}=(\mf{a}_i, n_i)$, then $R_{\mf{a}_i}(P_n)$ is a unit outside $\mf{p}$ again by Lemma \ref{lem15} because $R_{\mf{a}_i}(P_m)$  ($m=1,2,\ldots$) is norm compatible in the tower $F_\infty$ over $F$ by Corollary \ref{cor1}. If $\mf{P}\mid \mf{p}$ is a prime of $F_n$, we have $\ord_\mf{P}(x(P_n))<0$ but $\ord_\mf{P}(x(U))\geqs 0$ for any $U\in E_{\mf{a}_i}\backslash \{\o\}$, giving $\ord_\mf{P}(x(P_n)-x(U))=\ord_\mf{P}(x(P_n))$. Recalling that $\ord_\mf{P}(c_E(\mf{a}_i))=0$, we have $\ord_\mf{P}(R_{\mf{a}_i}(P_n))=\frac{1}{2}(\mr{N}\mf{a}_i-1)\ord_\mf{P}(x(P_n))$, because $(E_{\mf{a}_i}\backslash\{\o\})/\{\pm 1\}$ has order $\frac{1}{2}(\mr{N}\mf{a}_i-1)$. Hence
\begin{align*}\ord_\mf{P}(R_{\mc{D}}(P_n))&=\frac{1}{2}\ord_\mf{P}(x(P_n))\sum_{i}n_i(\mr{N}\mf{a}_i-1)=0,
\end{align*}
since $\sum_{i}n_i(\mr{N}\mf{a}_i-1)=0$ by the definition of $\mc{D}$. It follows that  $R_{\mc{D}}(P_n)$ is a unit. 
\end{proof}

Let $H_n=F_n\cap H_\infty$. For $n\geqs 1+e$ with $e=0$ or $1$ according as $p$ is odd or even, we have $[H_n:H]=p^{n-1-e}$. Note in particular that $H_n=H$ for $0\leqs n<1+e$. 
Furthermore, for each $n\geqs 0$, the classical theory of complex multiplication shows that $F_n=H(E_{\mathfrak{p}^n})$ contains the field $K(\mathfrak{p}^n)$, the ray class field of $K$ modulo $\mathfrak{p}^n$. Then if $p=2$, we have $H_n=K(\p^n)$ and
\[H_\infty=K(\mathfrak{p}^\infty)=\bigcup_n K(\mathfrak{p}^n)\]
is a $\Z_p$-extension of $H$. We identify $\Gamma$ with $\mathrm{Gal}(H_\infty/H)$. If $p>2$, $[K(\p^n):H]=\frac{1}{2}p^{n-1}(p-1)$ and $K(\p^n)$ strictly contains $H_n$.

Let $U_{H_n}$ denote the group of semi-local units of $H_n\otimes_K K_\p=\oplus_{\mathfrak{P}\mid \mathfrak{p}}H_{n,\mathfrak{P}}$ which are congruent to $1$ modulo the primes above $\mathfrak{p}$. We denote by $U_{H_\infty}$ the projective limit of the groups $U_{H_n}$ with respect to the norm maps. Similarly denote by $U_{F_n}$ and $U_{F_\infty}$ the corresponding objects for $F_n$ and $F_\infty$. Let $\mf{R}_\mc{D}^\sigma(P_n^{(\sigma)})=N_{K(\p^n)/H_n}R_\mc{D}^\sigma(P_n^{(\sigma)})$. In particular, $\mf{R}_\mc{D}^\sigma(P_n^{(\sigma)})=R_\mc{D}^\sigma(P_n^{(\sigma)})$ if $p=2$. 

\begin{defn} \begin{enumerate}
\item Define the group $\mc{C}_{H_n}$ to be the group generated by $\mf{R}_\mc{D}^\sigma(P_n^{(\sigma)})$ for all $\sigma\in G$, as $\mc{D}$ runs over the index set $I$. 
\item
We let $\bar{\mc{C}}_{H_n}$ denote the closure of $\mc{C}_{H_n}$ in $U_{H_n}$, and define the group of elliptic units
\[\bar{\mc{C}}_{H_\infty}=\varprojlim \bar{\mc{C}}_{H_n}\sb U_{H_\infty}\]
where the inverse limit is taken with respect to the norm maps.
\end{enumerate}
\end{defn}
Note that $\mc{C}_{H_n}$ is stable under the action of $\Gal(H_n/K)$, and does not depend on the choice of $P_n$. Note also that the roots of unity in $H_n$ are just $\{\pm 1\}$.

Given $\mf{u}=(u_n)\in U_{F_\infty}$, let $g_\mf{u}(W)\in \o_F\otimes_\o \o_\mf{p}[[W]]$ denote the Coleman power series  of $\mf{u}$ (see \cite[Theorem I.2.2]{dS} for more details), where $\o_F$ denotes the ring of integers of $F$. We write
\[\widetilde{\log}\: g_\mf{u}(W)=\log \: g_\mf{u}(W)-\frac{1}{p}\sum_{\omega\in\mc{D}_{\sigma,\mf{p}}}\log g_\mf{u}(W[+]\omega),\]
where we recall that $\mc{D}_{\sigma,\mf{p}}=\widehat{E}^\sigma_\mf{p}$ can be identified with $E^\sigma_\mf{p}$. It is well-known \cite[Lemma I.3.3]{dS} that $\widetilde{\log} \: g_\mf{u}(W)$ has integral coefficients. Define
\begin{equation*}i: U_{F_\infty}\to \Lambda_\ms{I}(\Gal(F_\infty/K))
\end{equation*}
by
\[\mf{u}\mapsto \mu_\mf{u}:=\prod_{\chi\in G^*}\sum_{\sigma\in G}\chi(\sigma)\varphi_K(\mf{s})^{-k}\mu_{\mf{u},\sigma},\]
where $\mu_{\mf{u},\sigma}$ is the measure satisfying
\begin{equation}\label{eq4.3.9}\widetilde{\log} \: g_\mf{u}\circ \delta_{\sigma,v}(W)=\int_{\mf{H}}(1+W)^{\chi_\mf{p}(\tau)}d\mu_{\mf{u},\sigma}(\tau).
\end{equation}
This induces an injective pseudo-isomorphism (\cite[Proposition III.1.3]{dS})
\begin{equation}\label{mapi} i: U_{F_\infty}\hat{\otimes}_{\Z_p}\ms{I}\to \Lambda_\ms{I}(\Gal(F_\infty/K)).
\end{equation}
Let $\mf{u}_\mc{D}=(R_\mc{D}^\sigma(P_n^{(\sigma)}))$. Then by construction, $\widetilde{\log} \: g_{\mf{u}_\mc{D}} = C_\mc{D}^\sigma$ where $C_\mc{D}^\sigma$ is defined in Lemma \ref{lem4.1.4}, and thus $i(\mf{u}_\mc{D})=\nu_\mc{D}=\prod_{\chi\in G^*}\sum_{\sigma\in G}\chi(\sigma)\varphi_K(\mf{s})^{-k}\mu_{\mc{D},\sigma}$. \\

\subsection{Statement of the Main Conjecture for $H_\infty/H$}\label{section4.4}~
\vspace{5 pt}\\
From now on, we always assume that $(p,h)=1$, where $h$ denotes the class number of $K$. This implies that $H_\infty/H$ is totally ramified at all primes above $\mf{p}$, since $K_\infty/K$ is totally ramified at all primes above $\mf{p}$.

Denote by $M(H_\infty)$ the maximal abelian $p$-extension of $H_\infty$ unramified outside the primes of $H_\infty$ above $\p$, and write
\[X(H_\infty)=\mathrm{Gal}(M(H_\infty)/H_\infty).\]

For every $n\geqs 0$, let $\mathcal{E}_{H_n}$ be the group of global units of $H_n$, and let $\bar{\mc{E}}_{H_n}$ be the closure of $\mathcal{E}_{H_n}\cap U_{H_n}$ in $U_{H_n}$ in the $p$-adic topology. Then we define $\bar{\mc{E}}_{H_\infty}=\varprojlim \bar{\mc{E}}_{H_n}$, where the inverse limits are taken with respect to the norm maps. A standard result from global class field theory says that the Artin map induces a $\Gal(H_n/K)$-isomorphism $U_{H_n}/\bar{\mc{E}}_{H_n}\simeq \Gal(M(H_n)/L(H_n))$, where $M(H_n)$ is the maximal abelian $p$-extension of $H_n$ unramified outside of the primes of $H_n$ above $\p$, and $L(H_n)$ is the maximal unramified abelian $p$-extension of $H_n$. Hence, writing $X(H_n)=\Gal(M(H_n)/H_n)$, we have an exact sequence, and taking the projective limit over $n$, we obtain
\begin{equation}\label{eq4.3} 0\to U_{H_\infty}/\bar{\mc{E}}_{H_\infty}\to X(H_\infty)\to \mathrm{Gal}(L(H_\infty)/H_\infty)\to 0,
\end{equation}
where $L(H_\infty)=\cup_{n\geqs 0} L(H_n)$ is the maximal unramified abelian $p$-extension of $H_\infty$.

Let $A(H_n)$ denote the $p$-primary part of the ideal class group of $H_n$, and let $A(H_\infty)$ denote the inductive limit of $A(H_n)$ taken with respect to the natural maps coming from the inclusion of fields. Class field theory identifies $A(H_\infty)$ with $\Gal(L(H_\infty)/H_\i)$. Thus we obtain the fundamental exact sequence needed for the proof of the main conjecture:
\begin{equation}\label{eq4.2}
0\to \bar{\mc{E}}_{H_\infty}/\bar{\mc{C}}_{H_\infty}\to U_{H_\infty}/\bar{\mc{C}}_{H_\infty}\to X(H_\infty)\to A(H_\infty)\to 0.
\end{equation} 
Recall that $\mathscr{G}=\Gal(H_\infty/K)$. Then we have
\[\ms{G}=G\times \Gamma\]
so that characters of $G$ can naturally be considered as characters of $\ms{G}$. 

\begin{lem}\label{lem4.3}We have 
\[i(\bar{\mc{C}}_{H_\infty}\hat{\otimes}_{\Z_p}\ms{I})=I_\ms{I}(\mathscr{G})\cdot\nu_\mf{p},\]
where $i$ denotes the map in \eqref{mapi} and $I_\ms{I}(\mathscr{G})$ denotes the augmentation ideal of $\Lambda_\ms{I}(\mathscr{G})$.
\end{lem}

\begin{proof} Recall that $i(\mf{u}_\mc{D})=\nu_\mc{D}=\theta_\mc{D}\nu_\mf{p}$. Hence we just need to show that $I_\ms{I}(\mathscr{G})\Lambda_\ms{I}(\mathscr{G})$ is generated by $\theta_\mc{D}$, $\mc{D}\in I$. In Lemma \ref{lem2}, we have  found $\mc{D}\in I$ such that $\theta_\mc{D}|_\Gamma$ generates $I_\ms{I}(\Gamma)$. It follows that for every $\chi\in G^*$, we have
\[i\left((\bar{\mc{C}}_{H_\infty}\hat{\otimes}_{\Z_p}\ms{I})^\chi\right)=\left(I_\ms{I}(\ms{G})\cdot\nu_\mf{p}\right)^\chi.\]
 The result now follows since we have an isomorphism $\ms{I}[[\mathscr{G}]]\simeq \ms{I}[[\Gamma]][G]$ and the decomposition $I_\ms{I}(\ms{G})=\oplus_{\chi\in G^*}I_\ms{I}(\ms{G})^\chi$, where $I_\ms{I}(\mathscr{G})^\chi=e_\chi I_\ms{I}(\Gamma)$ and $I_\ms{I}(\Gamma)$ is the augmentation ideal of $\Lambda_\ms{I}(\Gamma)$, which is generated by $\gamma-1$. 
\end{proof}

Define $\bold{\varphi}=I_\ms{I}(\ms{G})\nu_\mf{p}\subset \Lambda_\ms{I}(\ms{G})$. The following is an immediate consequence of the last two results.

\begin{thm}We have an exact sequence of $\Lambda_\ms{I}(\ms{G})$-modules
\[0\to \left(U_{H_\infty}/\bar{\mc{C}}_{H_\infty}\right)\hat{\otimes}_{\Z_p}\ms{I}\to\Lambda_\ms{I}(\mathscr{G})/\bold{\varphi} \to D\to 0,\]
where $D$ is finite.
\end{thm}

Given a finitely generated torsion $\Z_p[[\mathscr{G}]]$-module $X$, recall that $\mathrm{char}\left(X^\chi\right) \sb \Lambda_\ms{I}(\mathscr{G})^\chi$ denotes the characteristic ideal of the $\Lambda_\ms{I}(\mathscr{G})^\chi$-module $(X\hat{\otimes}_{\Z_p}\ms{I})^\chi$.

\begin{cor}\label{cor3.1} For every $\chi\in G^*$, we have
\[\mathrm{char}\left((U_{H_\infty}/\bar{\mc{C}}_{H_\infty})^\chi\right)=\bold{\varphi}^\chi.\]
\end{cor}

We are now ready to state the main conjecture for $H_\infty/H$ whose proof is in Chapter \ref{ch6}.

\begin{thm}[Main Conjecture for $H_\infty/H$]\label{mc} For every $\chi\in G^*$, we have
\[\mathrm{char}\left(X(H_\infty)^\chi\right)=\bold{\varphi}^\chi.\]
\end{thm}

Before we move on, we will verify that Theorem \ref{mc} holds for $p=2$ and $E=X_0(49)$, which is equal to the case $E=A(q)$ with $q=7$. In this case, we have $M(H_\infty)=H_\infty$, because the maximal abelian extension of $K$ in $M(H_\infty)$ coincides with the union $\cup _n K(\mf{p}^n)$ of ray class fields $K$ modulo $\mf{p}^n$. Thus $X(H_\infty)=0$, and it follows that Theorem \ref{mc} holds if and only if $\bold{\varphi}$ is a unit. By Theorem \ref{thm4.2.7}, this holds if and only if $(\chi_\mf{p}(\gamma)^2 -1)L(\ovl{\psi}_{E/H}^2,2)/\O_\i(E/H)^2$ is a unit at $\mf{p}$. It is easy to check that this holds, as discussed already in the introduction.
~\\
\section{Euler Systems}\label{ch5}

\subsection{Euler Systems of the Elliptic Units}~
\vspace{5 pt}\\
Let $\ms{G}_n=\Gal(H_n/K)$ and $\Lambda_n=\mb{Z}_p[\ms{G}_n]$. We denote by $\Lambda(\ms{G})=\mb{Z}_p[[\ms{G}]]=\varprojlim \Lambda_n$ the Iwasawa algebra of $\ms{G}$. In this section, we will treat the $\ms{G}$-modules occurring in the fundamental exact sequence \ref{eq4.2} as $\Lambda(\ms{G})$-modules. In fact, they are finitely generated and torsion as $\Lambda(\ms{G})$-modules. Given a finitely generated torsion $\Lambda(\ms{G})$-module $X$, write $\mathrm{char}_\Lambda(X)$ for the characteristic ideal of $X$ given by the structure theorem for finitely generated torsion $\Lambda(\ms{G})^\chi\simeq\Z_p[[\Gamma]]$-modules, and $\mathrm{char}_\Lambda\left(X^\chi\right)$ for the characteristic ideal of $X^\chi$ as a $\Lambda(\ms{G})^\chi$-module. This means that $\mathrm{char}_\Lambda(X\hat{\otimes}_{\Z_p}\ms{I})=\mathrm{char}(X)$. The aim of this chapter is to define and study Euler systems of the elliptic units $\bar{\mc{C}}_{H_\infty}$, defined in Chapter \ref{ch4}, for the tower $H_\infty/H$. The method of Euler systems we follow is due to Rubin \cite[Chapter 1]{rubin}. When combined with an application of \v{C}ebotarev density theorem, the results in this chapter enable us to prove a divisibility relation analogous to \cite[Theorem 8.3]{rubin}:
\[\mathrm{char}_\Lambda(A(H_\infty)) \text{ divides } p^k\mathrm{char}_\Lambda(\bar{\mc{E}}_{H_\infty}/ \bar{\mc{C}}_{H_\infty}),\]
for an integer $k\geqs 0$ ($k=0$ when $p>2$). 

Fix an integer $\ell> 1$. Let $\mc{I}_\ell$ be the set of squarefree ideals of $\o$ which are divisible only by primes $\q$ of $K$ such that
\begin{enumerate}[(i)]
\item $\q$ splits completely in $H_n/K$, and
\item $\mathrm{N}\q\equiv 1\bmod p^{\ell+e}$, where $e=0$ or $1$ according as $p$ is odd or even.
\end{enumerate}

Recall that $K(\mf{q})$ denotes the ray class field of $K$ modulo $\mf{q}$. In the following lemma, we define the field $H_n(\mf{q})$.
\begin{lem}Given a prime $\q\in \mc{I}_\ell$, we have a unique (cyclic) extension $H_n(\q)$ of $H_n$ of degree $p^\ell$ inside $H_n K(\q)$. Furthermore, $H_n(\q)/H_n$ is totally ramified at the primes above $\q$, and unramified everywhere else.
\end{lem}

\begin{proof}Since $\q$ is unramified in $H_n/K$, we have $K(\q)\cap H_n=H\cap H_n=H$. Hence, we have
\[\Gal(H_n K(\q)/H_n)=\Gal(K(\q)/H),\]
which isomorphic to $(\o/\q\o)^\times/\#(\widetilde{\bold{\mu}}_K)$ via the Artin map, where $\widetilde{\bold{\mu}}_K$ denotes the image of $\bold{\mu}_K$ under reduction modulo $\q$. Since $(\mf{q},2)=1$, the reduction modulo $\mf{q}$ map is injective, and this is cyclic of order $(\mr{N}\mf{q}-1)/(\#(\bold{\mu}_K))$ where $\#(\bold{\mu}_K)=2$. Hence it has a unique subgroup of order $p^\ell$ since  $\mr{N}\mf{q}\equiv 1\bmod p^{\ell+e}$, where $e=0$ or $1$ if $p>2$ or $p=2$. Furthermore, $H_nK(\q)/H_n$ is totally ramified at the primes above $\q$ and unramified everywhere else, so the assertions of the lemma follow.
\end{proof}

 If $\mf{r}=\prod_{i_1}^l \q_i\in \mc{I}_\ell$, we write $H_n(\mf{r})$ for the composite $H_n(\q_1)\cdots H_n(\q_l)$, and put $H_n(\o)=H_n$.

\begin{defn} Fix an integer $\ell>1$. An Euler system relative to $\ell$ is a collection of global units
\[\bold{\alpha}=\{\bold{\alpha}^\sigma(n,\mf{r}): \; n\geqs 1+e,\;  \mf{r}\in \mc{I}_\ell, \;\sigma\in G\}\] 
satisfying
\begin{enumerate}[(i)]
\item $\bold{\alpha}^\sigma(n,\mf{r})$ is a global unit of $H_n(\mf{r}),$
\item If $\mf{q}$ is a prime such that $\mf{rq}\in \mc{I}_\ell$, then
\begin{equation}\label{es2}\N_{H_n(\mf{rq})/H_n(\mf{r})}(\bold{\alpha}^\sigma(n,\mf{rq}))=\bold{\alpha}^\sigma(n,\mf{r})^{1-\Frob_{\mf{q}}^{-1}}
\end{equation}
where $\Frob_\mf{q}$ is the Frobenius of $\mf{q}$ in $\Gal(H_n(\mf{rq})/K)$.
\item \begin{equation}\label{es3}\N_{H_{n+1}(\mf{r})/H_n(\mf{r})}(\bold{\alpha}^\sigma(n+1,\mf{r}))=\bold{\alpha}^{\sigma\sigma_\mf{p}}(n,\mf{r}),
\end{equation}
where $\sigma_\mf{p}=(\mf{p}, H/K)\in G$.
\end{enumerate}
\end{defn}

The elliptic units give rise to an Euler system.  The following is \cite[Corollary 7.7]{rubin2}.

\begin{prop}\label{prop4.1}Suppose $\mf{m}$ is an ideal of $\o$ prime to $\mf{a}\mf{f}$, $P\in E_\mf{m}$ is a primitive $\mf{m}$-division point of $E$ and $\mf{r}$ is a prime ideal of $K$ dividing $\mf{m}$ , say $\mf{m}=\mf{m}'\mf{r}$. Then 
\[
\mr{N}_{K(\mf{m})/K(\mf{m}')}R_\mf{a}(P)=\left\{ \begin{array}{ll}
R_\mf{a}^{\sigma_\mf{r}}\left(\lambda_{E}(\mf{r})(P)\right)^{1-\Frob_\mf{r}^{-1}}&\mbox{  if $\mf{r}\nmid \mf{m}'$} \\
R_\mf{a}^{\sigma_\mf{r}}\left(\lambda_{E}(\mf{r})(P)\right) &\mbox{  if $\mf{r}\mid \mf{m}'$.}
       \end{array} \right.
\]
where $\Frob_\mf{r}$ denotes the Frobenius of $\mf{r}$ in $\Gal(K(\mf{m}')/K)$, and $\mr{N}_{K(\mf{m})/K(\mf{m}')}$ denotes the norm map from $K(\mf{m})$ to $K(\mf{m}')$.
\end{prop}

\begin{prop}\label{prop13} For all positive integers $m\geqs n$, we have
\[\mathrm{N}_{H_m/H_n}\bar{\mc{C}}_{H_m}=\bar{\mc{C}}_{H_n},\]
where $\mathrm{N}_{H_m/H_n}$ denotes the norm map from $H_{m}$ to $H_{n}$.
\end{prop}

\begin{proof}
By Corollary \ref{cor1}, we have $\mr{N}_{H_m/H_n}\mf{R}_\mc{D}(P_m)=\mf{R}_\mc{D}^{\sigma_\mf{p}^{m-n}}(P_n^{(\sigma_\mf{p}^{m-n})})$. Hence we have
\[\mr{N}_{H_m/H_n}\mc{C}_{H_m}=\mc{C}_{H_n}.\]
This completes the proof of the proposition.
\end{proof}

\begin{prop}\label{prop11}If $u\in\mc{C}_{H_n}$, then there exists an Euler system $\bold{\alpha}$ with $\bold{\alpha}^\sigma(n,1)=u$.
\end{prop}

\begin{proof} It suffices to consider the case $u=\mf{R}_\mc{D}^\sigma(P_n^{(\sigma)})$. Given $\mf{r}\in \mc{I}_\ell$, define $\alpha^\sigma_n(\mf{r})=\mf{R}_\mc{D}^\sigma\left(\lambda_{E^\sigma}(\mf{r})^{-1}(P_n^{(\sigma)})\right)$. Then clearly $\alpha^\sigma_n(1)=u$ and $\alpha^\sigma_n(\mf{r})$ is a global unit in $H_n(\mf{r})$. Furthermore, if $\mf{q}$ is a prime in $\mc{I}_\ell$ and $\mf{r}\mf{q}\in \mc{I}_\ell$, then $\sigma_\mf{q}=1$, so by Proposition \ref{prop4.1} we have
\begin{align*}N_{H_n(\mf{rq})/H_n(\mf{r})}(\alpha^\sigma_n(\mf{r}\mf{q}))&=\mf{R}_\mc{D}^\sigma\left(\lambda_{E^\sigma}(\mf{r})^{-1}(P_n^{(\sigma)})\right)^{1-\Frob_\mf{q}^{-1}}=\alpha^\sigma_n(\mf{r})^{1-\Frob_\mf{q}^{-1}},
\end{align*}
 and similarly
\begin{align*}N_{H_{n+1}(\mf{r})/H_n(\mf{r})}(\alpha^\sigma_{n+1}(\mf{r}))=\mf{R}_\mc{D}^{\sigma\sigma_\mf{p}}\left(\lambda_{E^{\sigma\sigma_\mf{p}}}(\mf{r})^{-1}(P_{n}^{(\sigma\sigma_\mf{p})})\right)=\alpha_n^{\sigma\sigma_\mf{p}}(\mf{r}).
\end{align*}
Therefore, defining $\bold{\alpha}^\sigma(n,\mf{r})=\alpha_n^\sigma(\mf{r})$ gives the result.
\end{proof}
~\\
\subsection{An Application of the \v{C}ebotarev Density Theorem}~
\vspace{5 pt}\\
Let $\ell>1$ be a fixed integer. For every prime $\q\in \mc{I}_\ell$, write $G_\q=\Gal(H_n(\mf{q})/H_n)$. Then $G_\q$ is cyclic of order $p^\ell$ so we fix a generator $\tau_\q$.  Fix $n\geqs 1+e$, and let 
\[I_{H_n}=\oplus_\mf{Q}\Z \mf{Q}\]
denote the group of fractional ideals of $H_n$ written additively, where the sum runs over the prime ideals of $H_n$. For every prime $\mf{q}$ of $K$, let
\[I_\mf{q}=\oplus_{\mf{Q}\mid \mf{q}}\Z \mf{Q}=\mathbb{Z}[\ms{G}_n]\mf{Q}.\]
For $y\in H_n^\times$ let $(y)_\mf{q}$, $[y]$ and $[y]_\mf{q}$ be the projection of the principal ideal $(y)$ in $I_\mf{q}$, $I_{H_n}/p^\ell I_{H_n}$ and $I_\mf{q}/p^\ell I_\mf{q}$ respectively. Note that  $[y]$ and $[y]_\mf{q}$ are well-defined for $y\in H_n^\times/(H_n^\times)^{p^\ell}$.

Suppose now that $\mf{Q}$ is a prime of $H_n$ lying above a prime $\mf{q}\in \mc{I}_\ell$, and we let $\widetilde{\mf{Q}}$ be the prime of $H_n(\mf{q})$ above $\mf{Q}$. Suppose $x\in H_n(\mf{q})^\times$ and $\rho\in G_\mf{q}$. Then $x^{1-\rho}\bmod \widetilde{\mf{Q}}\in (\o_{H_n(\mf{q})}/\widetilde{\mf{Q}})^\times$, where $\o_{H_n(\mf{q})}$ denotes the ring of integers of $H_n(\mf{q})$. We let $x^{1-\rho}\bmod \mf{Q}$ denote the image of $x^{1-\tau_\mf{q}}$ in $(\o_{H_n}/\mf{Q})^\times$, and write $(\ovl{x^{1-\tau_\mf{q}}})^{1/d}$ for the unique $d$-th root of the image of $x^{1-\tau_\mf{q}}$ in $(\o_{H_n}/\mf{Q})^\times/((\o_{H_n}/\mf{Q})^\times)^{p^\ell}$, where $d=(\mr{N}\mf{q}-1)/p^{\ell}$. Then the map
\[H_n(\mf{q})\to (\o_{H_n}/\mf{Q})^\times/((\o_{H_n}/\mf{Q})^\times)^{p^\ell}, \;\;\; x\to (\ovl{x^{1-\tau_\mf{q}}})^{1/d}\]
is surjective, with kernel $\{x\in H_n (\mf{q})^\times : \ord_{\widetilde{\mf{Q}}} (x)\equiv 0 \bmod p^\ell\}$. Let $w$ be the image of $x$ under this map. Then setting
\[l_\mf{Q}: (\o_{H_n}/\mf{Q})^\times/((\o_{H_n}/\mf{Q})^\times)^{p^\ell}\xto{\sim} \Z/p^\ell\Z, \;\;\; w\to \ord_{\widetilde{\mf{Q}}}(x) \bmod p^\ell\]
gives an isomorphism. Now define a map
\[\varphi_{\mf{q}}: (\o_{H_n}/\mf{q}\o_{H_n})^\times/((\o_{H_n}/\mf{q}\o_{H_n})^\times)^{p^\ell}\to I_\mf{q}/p^\ell I_\mf{q}\]
by $\varphi_\mf{q}(w)=\sum_{\mf{Q}\mid \mf{q}}l_\mf{Q}(w)\mf{Q}$, where we also write $l_\mf{Q}$ for the map composed with the natural projection $(\o_{H_n}/\mf{q}\o_{H_n})^\times/((\o_{H_n}/\mf{q}\o_{H_n})^\times)^{p^\ell}\to (\o_{H_n}/\mf{Q})^\times/((\o_{H_n}/\mf{Q})^\times)^{p^\ell}$.

\begin{prop}\label{prop7} Suppose $\bold{\alpha}=\{\bold{\alpha}^\sigma(n,\mf{r}): \; n\geqs 1+e,\;  \mf{r}\in \mc{I}_\ell, \;\sigma\in G\}$ is an Euler system. Given $\sigma\in \Gal(H/K)$, there exists a canonical map
\[\kappa_{\bold{\alpha}}: \mc{I}_\ell\to H_n^\times/(H_n^\times)^{p^\ell}\]
such that for every $n\geqs 1$ and $\mf{r}\in \mc{I}_\ell$ we have $\kappa_\a(\mf{r})=\bold{\alpha}^\sigma(n, \mf{r})^{\mc{D}_\mf{a}}\bmod (H_n(\mf{r})^\times)^{p^\ell}$, where $\mc{D}_\mf{a}=\prod_{\q\mid \mf{a}}\sum_{i=0}^{p^\ell-1}i\tau_\q^i$.
Furthermore, if $\mf{r}\in \mc{I}_\ell$ and $\mf{q}$ a prime of $K$, then
\begin{enumerate}[(i)]
\item If $\mf{q}\nmid \mf{r}$ then $[\kappa_{\bold{\alpha}}(\mf{r})]_{\mf{q}}=0$.
\item If $\mf{q}\mid\mf{r}$ then $[\kappa_{\bold{\alpha}}(\mf{r})]_{\mf{q}}=\varphi_\mf{q}(\mf{r}/\mf{q})$.
\end{enumerate}
\end{prop}

\begin{proof}This follows easily from an alternative definition of Euler systems as a Galois equivariant map which takes values in $\cap_{n, \mf{r}} H_n(\mf{r})^\times$. See \cite[Proposition 2.2]{rubin} and \cite[Proposition 2.4]{rubin} for details.
\end{proof}

\begin{lem}\label{lem6} Let
\[\res: H_n^\times/(H_n^\times)^{p^\ell}\to H_n(\bold{\m}_{p^{\ell+e}})^\times/( H_n(\bold{\m}_{p^{\ell+e}})^\times)^{p^\ell},\]
be the natural map, where $e=0$ or $1$ if $p>2$ or $p=2$. Then $\res$ is injective if $p>2$, and $4\ker (\res)=0$ if $p=2$.
\end{lem}

\begin{proof} We have $H_n^\times/(H_n^\times)^{p^\ell}\simeq H^1(\ovl{H_n}/H_n, \bold{\m}_{p^\ell})$ and $ H_n(\bold{\m}_{p^{\ell+e}})^\times/( H_n(\bold{\m}_{p^{\ell+e}})^\times)^{p^\ell}\simeq H^1(\ovl{H_n(\bold{\m}_{p^{\ell+e}})}/H_n(\bold{\m}_{p^{\ell+e}}), \bold{\m}_{p^\ell})$ by Hilbert's Theorem 90. Hence $\ker (\res)=\linebreak H^1(\Gal(H_n(\bold{\m}_{p^{\ell+e}})/H_n),\bold{\m}_{p^\ell})$. Also, $H_\infty\cap K(\bold{\mu}_{p^\infty})=K$ because $\mf{p}$ and $\mf{p}^*$ are totally ramified in $K(\bold{\mu}_{p^\infty})/K$, but $H_\infty/K$ is unramified outside $\mf{p}$. It follows that $H_\i\cap \Q(\bold{\m}_{p^\i})=\Q$, and
\[\Gal(H_n(\bold{\m}_{p^{\ell+e}})/H_n)=(\Z/p^{\ell+e})^\times\simeq \D\times\Z/p^{\ell-1}\Z.\]
Here, $\D=\Gal(H_n(\bold{\m}_{p^{1+e}})/H_n)$ is cyclic of order $p-1$ or $p$ if $p$ is odd or even, and $\Gal(H_n(\bold{\m}_{p^{\ell+e}})/H_n(\bold{\m}_{p^{1+e}}))\simeq \Z/p^{\ell-1}\Z $.
So if $p>2$, $\Gal(H_n(\bold{\m}_{p^{\ell+e}})/H_n)$ is cyclic and we have $\ker (\res)=0$, as required. If $p=2$, taking the inflation-restriction sequence gives
\[0\to H^1(\D, \bold{\m}_4)\to \ker (\res)\to H^1(\Gal(H_n(\bold{\m}_{2^{\ell+1}})/H_n(\bold{\m}_{4}))\,\bold{\m}_{2^\ell}),\]
and $H^1(\D, \bold{\m}_4)=H^1(\Gal(H_n(\bold{\m}_{2^{\ell+1}})/H_n(\bold{\m}_{4}),\bold{\m}_{2^\ell})=\Z/2\Z$. Hence $\#\left(\ker (\res)\right)\mid 4$, and the result follows.
\end{proof}

We now employ the \v{C}ebotarev density theorem to obtain a prime of $H_n$ lying above that of $\mc{I}_\ell$ with properties desirable for the induction argument in Section \ref{section5.3}. Let $I(\mathscr{G})$ denote the augmentation ideal of $\Lambda(\mathscr{G})$.

\begin{thm}\label{thm14} Suppose $\chi\in G^*$. Let $v\in \left(H_n^\times/(H_n^\times)^{p^\ell}\right)^\chi$ and define $V$ to be the finite $\Lambda_n$-submodule of $(H_n^\times/(H_n^\times)^{p^\ell})^\chi$ generated by $v$. Finally, fix $\phi\in \Hom_{\Lambda_n}(V,\Lambda_n/p^\ell \Lambda_n)$, $\phi\neq 0$. Let $\mf{c}\in p^eI(\ms{G})A(H_n)^\chi$, where $e=1$ if $p=2$ and $e=0$ otherwise. Then  there is a prime $\mf{q}\in \mc{I}_\ell$ and a prime $\mf{Q}$ of $H_n$ above $\mf{q}$ such that
\begin{enumerate}[(i)]
\item the ideal class of $\mf{Q}$ in $A(H_n)^\chi$ is equal to $\mf{c}$,
\item $[v]_\mf{q}=0$ and there exists $r\in (\Z/p^\ell \Z)^\times$ such that $\varphi_\mf{q}(v)=p^{3e} r \phi(v)\mf{Q}$.
\end{enumerate}
\end{thm}

\begin{proof}
Write $H_n'=H_n(\bold{\mu}_{p^{\ell+e}})$, and $V_{\res}=V/V\cap \ker (\res)$, where $\res$ is the map in Lemma \ref{lem6}. Note that $V=V_{\res}$ if $p\neq 2$ by Lemma \ref{lem6}, and Kummer theory gives $\Gal(H_n'(v^{1/p^\ell})/H_n')\simeq \Hom(V_{\res}, \bold{\mu}_{p^\ell})$ . Fix a primitive $p^\ell$-th root of unity $\zeta$, and let $\iota: \Lambda_n/p^\ell \Lambda_n\to \bold{\mu}_{p^\ell}$ be the map sending $\sum a_\sigma \sigma \bmod p^\ell$ to $\zeta^{a_1}$. Define $\beta:=p^{3e}(\iota \circ \phi)$. Then by Lemma \ref{lem6}, we have $\beta\in p^e\Hom(V_{\res}, \bold{\mu}_{p^\ell})$. Let $b$ be the element of $p^e \Gal(H_n'(v^{1/p^\ell})/H_n')$ corresponding to $\beta$ via the Kummer map, so that $\beta(v)=\frac{b(v^{1/p^\ell})}{v^{1/p^\ell}}$. Let $L_n$ denote the unramified extension of $H_n$ such that $A(H_n)^\chi=\Gal(L_n/H_n)$. Then we see that there exists a submodule $W$ of $V_{\res}$ such that 
\[\Gal(L_n'/L_n\cap H_n')=\Gal(L_n'H_n'/H_n')=\Hom(W,\bold{\mu}_{p^\ell}),\]
where $L_n'=L_n\cap H_n'(v^{1/p^\ell})$. On the other hand, $\Gal(H_n'/H_n)$  acts trivially on $\Gal(L_n'H_n'/H_n')$ and $\bold{\mu}_{p^\infty}(H_n)=\bold{\mu}_{2}$, so that $\Hom(W,\bold{\mu}_{p^\ell})=\Hom(W, \bold{\mu}_{2})$. Therefore, $p^e\Gal(L_n'/L_n\cap H_n')=0$, and $b$ restricted to $L_n'$ is trivial. Furthermore, $I(\mathscr{G})$ annihilates $\Gal(L_n\cap H_n'/H_n)$ since $H_n'$ is abelian over $H$, so we can consider $\mf{c}$ as an element of $p^e \Gal(L_n/L_n')$. Hence we can choose $\rho\in \Gal(L_n H_n'(v^{1/p^\ell})/H_n)$ such that $\rho|_{L_n}=\mf{c}$ and $\rho|_{H_n'(v^{1/p^\ell})}=b$. By the \v{C}ebotarev density theorem, we can pick a prime $\mf{Q}$ of $H_n$ lying above a prime $\mf{q}\in \mc{I}_\ell$ whose Frobenius is equal to $\rho$. Then class field theory identifies $[\mf{Q}]\in A(H_n)^\chi$ with $\Frob_\mf{Q}\in\Gal(L_n/H_n)$, so (i) follows. Now, $[v]_\mf{q}=0$ because all primes lying above $\mf{q}$ are unramified in $H_n'(v^{1/p^\ell})/H_n$, and $v$ is a $p^\ell$-th power in $H_n'(v^{1/p^\ell})$. Also,
\begin{align*}\ord_\mf{Q}(p^{3e}\phi(v)\mf{Q})=0&\iff p^{3e}(\iota\circ\phi(v))=\frac{b((v)^{1/p^\ell})}{(v)^{1/p^\ell}}=1\\&\iff v \text{ is an $p^{\ell}$-th power modulo $\mf{Q}$.}
\end{align*}
On the other hand, $\ord_\mf{Q}(\varphi_\mf{q}(v))=l_\mf{Q}(v)=0$ if and only if $v$  is an $p^{\ell}$-th power modulo $\mf{Q}$. It follows that there exists $r\in (\Z/p^\ell \Z)^\times$ with $\ord_\mf{Q}(\varphi_\mf{q}(v))=r\ord_\mf{Q}(p^{3e}\phi(v)\mf{Q})$, and the map sending $v$ to $\varphi_\mf{q}(v)-p^{3e}r\phi(v)\mf{Q}$ gives rise to a $\ms{G}_n$-equivariant injective homomorphism from $V$ to $\oplus_{\substack{h\in \ms{G}_n \\ h\neq 1}}(\Z/p^\ell \Z)\mf{Q}^h$. But the latter has no non-zero $\ms{G}_n$-stable submodules.
\end{proof}
\vspace{5 pt}
\subsection{The Inductive Argument}\label{section5.3}~
\vspace{5 pt}\\
Let $e=0$ or $1$ if $p>2$ or $p=2$. For $n\geqs 1+e$, let $\Gamma_n=\Gamma^{p^{n-1-e}}$. Define $I(H_n)$ to be kernel of the restriction map $\Lambda(\ms{G})\to \Lambda_n$, that is, the ideal of $\Lambda(\ms{G})$ generated by $\{\sigma-1: \sigma\in \Gamma_n\}$.

\begin{prop}\label{thm11} $X(H_\infty)$ is a finitely generated torsion $\Lambda(\ms{G})$-module, and it has no non-zero finite submodule. Furthermore, $X(H_\infty)/I(H_n) X(H_\infty)$ is finite for any $n$.
\end{prop}

\begin{proof} The first statement follows from \cite[Lemma 13, Lemma 14]{coa2}. Iwasawa theory shows that $I(H_n)X(H_\infty)=\Gal(M(H_\infty)/M(H_n))$, because $M(H_n)$ is the largest abelian extension of $H_n$ inside $M(H_\infty)$. Hence we have an exact sequence
\begin{equation}0\to X(H_\infty)/I(H_n) X(H_\infty) \to X(H_n)\to \Gal(H_\infty/H_n)\to 0,
\end{equation}\label{eq5.1}
where $X(H_n)=\Gal(M(H_n)/H_n)$. Clearly the $\Z_p$-rank of $\Gal(H_\infty/H_n)$ is $1$, and the same is true for $X(H_n)$ since we have $\rank_{\Z_p}(X(H_n))=\rank_{\Z_p}(U_{H_n}/\bar{\mc{E}}_{H_n})$ by class field theory, and the right hand side is equal to $1$ by the $\mf{p}$-adic analogue of Leopoldt's conjecture which holds for abelian extensions of $K$.
\end{proof}

\begin{lem}\label{thm12}$\mathrm{char}_\Lambda(A(H_\infty))$ is prime to $I(H_n)$.
\end{lem} 

\begin{proof}$A(H_\infty)$ is a quotient of $X(H_\infty)$, so $A(H_\infty)/I(H_n) A(H_\infty)$ is a quotient of $X(H_\infty)/I(H_n) X(H_\infty)$. Since the latter is finite by Proposition \ref{thm11}, we also have that $A(H_\infty)/I(H_n) A(H_\infty)$ is finite. Let $f_1,\ldots, f_k\in \Lambda(\ms{G})$ be such that $\mathrm{char}_\Lambda (A(H_\infty))$ is given by $\left(\prod_{i=1}^k f_i\right)\Lambda(\ms{G})$.  It can then be shown using the snake lemma that $\prod_{i=1}^k \Lambda(\ms{G})/\left(I(H_n)+f_i\Lambda(\ms{G})\right)$ is finite. The result now follows.
\end{proof}

Let $\pi_U: U_{H_\infty}/I(H_n)U_{H_\infty}\to U_{H_n}$, $\pi_\mc{E}: \bar{\mc{E}}_{H_\infty}/I(H_n)\bar{\mc{E}}_{H_\infty}\to \bar{\mc{E}}_{H_n}$ and $\pi_\mc{C}: \bar{\mc{C}}_{H_\infty}/I(H_n)\bar{\mc{C}}_{H_\infty}\to \bar{\mc{C}}_{H_n}$ denote the maps induced by the projection map,  and let $D_\mf{p}=\prod_{\mathfrak{P}\mid\mathfrak{p}}D_\mathfrak{P}$ denotes the group generated by the decomposition groups $D_\mathfrak{P}$ of $\mf{P}$ in $H_\infty/H$. Write $I(D_\mf{p})$ for the ideal of $\Lambda(\ms{G})$ generated by $\{\sigma-1: \sigma\in D_\mathfrak{p}\}$.

\begin{lem}\label{thm8}\begin{enumerate}[(i)]\item $I(D_\mf{p})\ker \pi_U=I(D_\mf{p})\coker \pi_U=0$,
\item $I(D_\mf{p})\ker \pi_\mc{E}=0$,
\item There exists an ideal $\mc{B}$ of finite index in $\Lambda(\ms{G})$ such that 
\[I(D_\mf{p})\mc{B}\coker \pi_\mc{E}=0.\]
\end{enumerate}
\end{lem}

\begin{proof} For (i), see \cite[Theorem 5.1]{rubin}. By Proposition \ref{thm11}, $X(H_\infty)^{\Gamma_n}$ is a finite submodule of $X(H_\infty)$, and therefore is  equal to zero. It follows from \eqref{eq4.3} that $\left(U_{H_\infty}/\bar{\mc{E}}_{H_\infty}\right)^{\Gamma_n}=0$. Now \cite[Theorem 7.6 (i)]{rubin} easily applies to give an injection $\ker \pi_\mc{E}\to \ker \pi_U$, so assertion (ii) follows from (i). To prove assertion (iii), apply the snake lemma to \eqref{eq4.3} and use the fact that $X(H_\infty)^{\Gamma_n}=0$ to obtain $A(H_\infty)^{\Gamma_n}\simeq \ker \pi_{U/\mc{E}}$, where  $\pi_{U/\mc{E}}: \left( U_{H_\infty}/\bar{\mc{E}}_{H_\infty}\right)/I(H_n)\left(U_{H_\infty}/\bar{\mc{E}}_{H_\infty}\right)\to  U_{H_n}/\bar{\mc{E}}_{H_n}$ denotes the map induced by the projection map. Note that $A(H_\infty)^{\Gamma_n}$ is finite, since $A(H_\infty)/I(H_n)A(H_\infty)$ is finite. Assertion (iii) now follows from (i) and (ii) on taking $\mc{B}$ to be the annihilator of the maximal finite submodule of $A(H_\infty)$ in $\Lambda(\ms{G})$. 
\end{proof}

\begin{lem}\label{thm4.12}
$\rank_{\Lambda(\mathscr{G})}(\bar{\mc{C}}_{H_\infty})=1$ and $\coker(\pi_\mc{C})=\ker(\pi_\mc{C})=0$.
\end{lem}

\begin{proof} By Lemma \ref{lem4.3}, there is a isomorphism of $\Lambda(\ms{G})$-modules
\[\bar{\mc{C}}_{H_\infty}\simeq I(\mathscr{G}),\]
where $I(\mathscr{G})$ is the augmentation ideal of $\Lambda(\mathscr{G})$, so the first statement follows on noting that $\rank_{\Lambda(\mathscr{G})}(\Lambda(\mathscr{G})/I(\mathscr{G}))=\rank_{\Lambda(\mathscr{G})}(\Z_p)=0$. By Proposition \ref{prop13}, the projection map $\pi_\mc{C}$ is surjective, so $\coker \pi_\mc{C}=0$. Now, the first statement of the theorem gives $\bar{\mc{C}}_{H_\infty}/I(H_n) \bar{\mc{C}}_{H_\infty}\simeq \Lambda_n$ as $\Lambda(\ms{G})$-modules. Furthermore, $\bar{\mc{C}}_{H_n}$ is isomorphic to a submodule $Y$ of finite index in $\Lambda_n$. Define a map $f: \Lambda_n\to Y$ so that it commutes with the map $\pi_\mc{C}$. Then clearly $\ker \pi_\mc{C}\sb \ker f$ and $\coker f$ is a quotient of  $\coker \pi_\mc{C}$, which is equal to zero. Thus $\ker f$  is finite, and hence equal to zero since $\Lambda_n$ has no non-zero finite submodules. The theorem now follows.
\end{proof}

In the next result, we obtain a map $\theta_{\lambda,n}: \bar{\mc{E}}_{H_n}\to \Lambda_n$ which allows us to relate a generator of $\mathrm{char}_\Lambda(\bar{\mc{E}}_{H_\infty}/\bar{\mc{C}}_{H_\infty})$ to the image under $\theta_{\lambda,n}$ of an element of $\bar{\mc{C}}_{H_n}$.

\begin{cor}\label{cor4}$\mathrm{char}_\Lambda(\bar{\mc{E}}_{H_\infty}/\bar{\mc{C}}_{H_\infty})$ is prime to $I(H_n)$. Furthermore, there exists an ideal $\mc{B}\sb \Lambda(\mathscr{G})$ such that for every $\lambda\in I(\ms{G})\mc{B}$, there is a map $\theta_{\lambda,n}: \bar{\mc{E}}_{H_n}\to \Lambda_n$ satisfying
\[\lambda^2\mathrm{char}_\Lambda(\bar{\mc{E}}_{H_\infty}/\bar{\mc{C}}_{H_\infty})\Lambda_n\sb \theta_{\lambda, n}(\bar{\mc{C}}_{H_n}).\] 
\end{cor}

\begin{proof}
The first assertion follows from Lemma \ref{thm4.12} and the snake lemma. For the second assertion, we can adopt the proof of \cite[Corollary 7.10]{rubin} on letting $\mc{B}=\mc{A}_1\mc{A}_2$ where $\mc{A}_1$ satisfies Lemma \ref{thm8} (iii) and $\mc{A}_2$ is the annihilator of $\mathrm{char}_\Lambda(\bar{\mc{E}}_{H_\infty}/\bar{\mc{C}}_{H_\infty})/\theta(\bar{\mc{C}}_{H_\infty})$.
\end{proof}

Fix $f_1,\ldots, f_k\in \Lambda(\ms{G})$ so that $\mathrm{char}_\Lambda (A(H_\infty))$ is given by $\left(\prod_{i=1}^k f_i\right)\Lambda(\ms{G})$.  The next lemma is \cite[Proposition 6.5]{rubin}.

\begin{lem}\label{lem17} There exists an ideal $\mc{B}$ of finite index in $\Lambda(\ms{G})$ such that there exist classes $\mf{c}_1,\ldots \mf{c}_k\in A(H_n)$ satisfying $\mc{B}\mathrm{Ann}(\mf{c}_i)\sb f_i\Lambda_n$ for every $i$, where $\mathrm{Ann}(\mf{c}_i)\sb \Lambda_n$ is the annihilator of $\mf{c}_i$ in $A(H_n)/(\mf{c}_1 \Lambda_n + \cdots + \mf{c}_{i-1}\Lambda_n)$.

\end{lem}

Let $\ell>1$ be a fixed integer. Given a prime $\mf{Q}$ of $H_n$ lying above $\mf{q}\in \mc{I}_\ell$, $I_\mf{q}$ is a free $\mathbb{Z}[\ms{G}_n]$-module of rank $1$ generated by $\mathfrak{Q}$, and we define
\[v_\mf{Q}: H_n^\times\to \Lambda_n \text{\;\;\; by \;\;\;} v_\mf{Q}(w)\mf{Q}=(w)_\mf{q},\]
\[\bar{v}_\mf{Q}: H_n^\times/(H_n^\times)^{p^\ell}\to \Lambda_n/p^\ell \Lambda_n \text{\;\;\; by \;\;\;} \bar{v}_\mf{Q}(w)\mf{Q}=[w]_\mf{q}\]

The following lemma  is a combination of \cite[Lemma 8.2]{rubin} and \cite[Lemma 3.8.4]{gon} and will be used in the induction argument of Theorem \ref{thm9}.
\begin{lem}\label{lem18} Fix an integer $\ell> 1$. Suppose $\chi\in G^*$, $v\in \left(H_n^\times/(H_n^\times)^{p^\ell}\right)^\chi$, $\mf{q}\in\mc{I}_\ell$ is a prime, $\mf{Q}$ is a prime of $H_n$ lying above $\mf{q}$, $S$ is a set of primes of $K$ not containing $\mf{q}$, and $f, \lambda_0, \lambda_1, \lambda_2\in \Lambda(\ms{G})$, with $\lambda_0=2$ if $p=2$. Write $B_n$ for the subgroup of $A(H_n)$ generated by the primes of $H_n$ lying above the primes in $S$, $\mf{c}$ for the image of $\mf{Q}$ in $A(H_n)^\chi$ and $V$ for the $\Lambda_n$-submodule of $H_n^\times/(H_n^\times)^{p^\ell}$ generated by $v$. Suppose also that we have
\begin{enumerate}[(i)]
\item $[v]_\mf{r}=0$ for a prime $\mf{r}$ of $K$ not in $S\cup \{\mf{q}\}$,
\item the annihilator $\mathrm{Ann}(\mf{c})\sb \Lambda_n^\chi$ of $\mf{c}$ in $A(H_n)^\chi/B_n^\chi$ satisfies $\lambda_1\mathrm{Ann}(\mf{c})\sb f\Lambda_n^\chi$,
\item $\#(A(H_n)^\chi)\mid p^\ell$ and $\bar{v}_\mf{Q}(v)$ divides $(p^\ell/\#(A(H_n)^\chi))\lambda_2$ in $\Lambda_n^\chi/p^\ell \Lambda_n^\chi$, and
\item $f\Lambda(\ms{G})$ is prime to $I(H_n)$.
\end{enumerate}
Then there exists a $\ms{G}_n$-equivariant map $\phi: V\to \Lambda_n/p^\ell \Lambda_n$ satisfying
\[f\phi(v)=\lambda_0 \lambda_1 \lambda_2 \bar{v}_\mf{Q}(v).\]
\end{lem}

Finally, we make full use of the results from Chapter \ref{ch5} and apply induction to establish a divisibility relation between $\mathrm{char}_\Lambda(A(H_\infty)^\chi)$ and $\mathrm{char}_\Lambda(\bar{\mc{E}}_{H_\infty}/ \bar{\mc{C}}_{H_\infty})^\chi$.
\begin{thm}\label{thm9} Let $k$ be the number of $f_i$ appearing in $\mathrm{char}_\Lambda(A(H_\infty))$. 
\begin{enumerate}[(i)]
\item If $p>2$ and $\chi\in G^*$, we have
\[\mathrm{char}_\Lambda(A(H_\infty)^\chi)\text{\; divides \;}I(D_\mf{p})^{2k+2}\mathrm{char}_\Lambda(\bar{\mc{E}}_{H_\infty}/ \bar{\mc{C}}_{H_\infty})^\chi.\]
\item If $p=2$ and $\chi\in G^*$, we have
\[\mathrm{char}_\Lambda(A(H_\infty)^\chi) \text{\; divides \;}2^{6k+6}\mathrm{char}_\Lambda(\bar{\mc{E}}_{H_\infty}/ \bar{\mc{C}}_{H_\infty})^\chi.\]
\end{enumerate}
\end{thm}

\begin{proof}
 Fix a generator $\beta$ of $\mathrm{char}_\Lambda(\bar{\mc{E}}_{H_\infty}/ \bar{\mc{C}}_{H_\infty})^\chi$. Let $\mc{B}$ be an ideal of finite index in $\Lambda(\ms{G})$ satisfying the conditions in Lemma \ref{thm8} (iii) and Lemma \ref{lem17}. Take $\lambda\in I(D_\p)\mc{B}$ (or $2 I(\ms{G})\mc{B}$ if $p=2$). The existence of $\mc{B}$ is clear from the proof of Lemma 4.14. Pick $\ell> 1$ large enough so that we have
\begin{equation}\label{eq14}p^\ell \lambda \Lambda_n\sb p^{n+4ke}\lambda^{2k}\beta(\#(A(H_n)^\chi)))\Lambda_n.
\end{equation}
Here, $e=1$ or $0$ according as $p=2$ or $p\neq 2$. The choice of $\ell$ will be justified later in the proof. Now, by Corollary \ref{cor4}, there exists $\theta_{\lambda, n}: \bar{\mc{E}}_{H_n}\to \Lambda_n$ such that $\lambda^2 \beta\in \theta_{\lambda,n}(\bar{\mc{C}}_{H_n}^\chi)$.
Thus, we may fix $u\in \bar{\mc{C}}_{H_n}^\chi$ with $\theta_{\lambda,n}(u)=\lambda^2 \beta$, and also we fix $u_0\in \mc{C}_{H_n}^\chi$ with $u\equiv  u_0\bmod (\bar{\mc{C}}_{H_n}^\chi)^{p^\ell}.$
By Proposition \ref{prop11}, we have an Euler system $\bold{\alpha}$ and $\sigma\in G$ with $\bold{\alpha}^\sigma(n,1)=u_0$.
Let $\kappa_{\bold{\alpha}}$ be the map defined in Proposition \ref{prop7}, and let $\mf{c}_1,\ldots ,\mf{c}_k\in A(H_n)$ be as given in Lemma \ref{lem17}. 
We will use induction to select primes $\mf{Q}_1,\ldots ,\mf{Q}_{k+1}$ of $H_n$ lying above primes $\mf{q}_1,\ldots,\mf{q}_{k+1}$ of $K$ satisfying:
\begin{equation}\label{eq12}[\mf{Q}_i]=\lambda\mf{c}_i^\chi \text{ in }A(H_n)^\chi\text{, and }\mf{q}_i\in \mc{I}_\ell,
\end{equation}
\begin{equation}\label{eq13} \bar{v}_{\mf{Q}_1}(\kappa_{\bold{\alpha}}(\mf{q}_1)^\chi)=r_1 p^{4e}\lambda^2\beta\text{ and } f_{i-1}\bar{v}_{\mf{Q}_i}(\kappa_{\bold{\alpha}}(\mf{a}_i)^\chi)=r_i p^{4e} \lambda^2\bar{v}_{\mf{Q}_{i-1}}(\kappa_{\bold{\alpha}}(\mf{a}_{i-1})^\chi),
\end{equation}
where $\mf{a}_i=\mf{q}_1\cdots \mf{q}_i$ and $r_i\in (\Z/p^\ell \Z)^\times$.

For $i=1$, we take $\mf{c}=\lambda\mf{c}_1^\chi\in p^e I(\ms{G})A(H_n)^\chi$, $W=\left(\bar{\mc{E}}_{H_n}/\bar{\mc{E}}_{H_n}\cap (H_n^\times)^{2^\ell}\right)^\chi$, $\phi=p^e\theta_{\lambda,n}$ and apply Theorem \ref{thm14} and Proposition \ref{prop7}. Then we obtain a prime $\mf{Q}_1$ of $H_n$ lying above a prime $\mf{q}_1\in \mc{I}_\ell$ such that $[\mf{Q}_1]=\lambda\mf{c}_1^\chi$ in $A(H_n)^\chi$ and $r_1\in (\Z/p^\ell \Z)^\times$ satisfying
\begin{align*}[(\kappa_{\bold{\alpha}}(\mf{q}_1)^\chi)]_{\mf{Q}_1}&=\varphi_{\mf{q}_1}(\kappa_{\bold{\alpha}}(1)^\chi)=\varphi_{\mf{q}_1}(\bold{\alpha}^\sigma(n,1))= r_1 p^{3e} \phi(u_0)\mf{Q}_1\\&= r_1 p^{4e} \theta_{\lambda,n}(u_0)\mf{Q}_1= r_1 p^{4e} \lambda^2 \beta \mf{Q}_1.
\end{align*}
Thus we have $\bar{v}_{\mf{Q}_1}\left(\kappa_{\bold{\alpha}}(\mf{q}_1)^\chi\right)= r_1 p^{4e} \lambda^2 \beta$  by the definition of $[\cdot ]_{\mf{Q}_1}$.

Now, let $1<i<k$ and suppose we have selected primes $\mf{Q}_1, \ldots ,\mf{Q}_i$ satisfying \eqref{eq12} and \eqref{eq13}. We will define $\mf{Q}_{i+1}$. Recall $\mf{a}_i=\prod_{j\leqs i}\mf{q}_j$.
Let $V_i$ be the $\Lambda_n$-submodule of $H_n^\times/(H_n^\times)^{p^\ell}$ generated by $\kappa_{\bold{\alpha}}(\mf{a}_i)^\chi$. We will apply Lemma \ref{lem18} with $\mf{Q}=\mf{Q}_i$, $v=\kappa_{\bold{\alpha}}(\mf{a}_i)^\chi$, $\lambda_1=\lambda_2=\lambda$ and $S=\{\mf{q}_1,\ldots, \mf{q}_{i-1}\}$. This is possible because conditions (i), (ii) and (iv) of Lemma \ref{lem18} are satisfied thanks to Proposition \ref{prop7}, Lemma \ref{lem17} and Lemma \ref{thm12}, and (iii) is satisfied because by \eqref{eq13}, $\bar{v}_{\mf{Q}_i}(\kappa_{\bold{\alpha}}(\mf{a}_i)^\chi)$ divides $p^{4ie}\lambda^{2i}\beta$ in $\Lambda_n^\chi/p^\ell \Lambda_n^\chi$, so by the choice of $\ell$ made in \eqref{eq14}, $\bar{v}_{\mf{Q}_i}(\kappa_{\bold{\alpha}}(\mf{a}_i)^\chi)$ divides $\left(p^\ell/\#(A(H_n)^\chi)\right)\lambda$ in $\Lambda_n^\chi/p^\ell \Lambda_n^\chi$. Thus, we obtain a map $\phi_i: V_i\to \Lambda_n/p^\ell \Lambda_n$ such that
\begin{equation}\label{eq15}f_i \phi_i(\kappa_{\bold{\alpha}}(\mf{a}_i)^\chi)=p^e\lambda^2\bar{v}_{\mf{Q}_i}(\kappa_{\bold{\alpha}}(\mf{a}_i)^\chi).
\end{equation}
Now, applying Theorem \ref{thm14} by setting $V=V_i$, $\mf{c}=\lambda\mf{c}_{i+1}^\chi$, $\phi=\phi_i$, we obtain a prime $\mf{q}_{i+1}\in \mc{I}_\ell$ and a prime $\mf{Q}_{i+1}$ of $H_n$ lying above it. Then (i) and (ii) of Theorem \ref{thm14} gives \eqref{eq12} for $i+1$. Furthermore, by Proposition \ref{prop7} (ii) and Theorem \ref{thm14} (ii), for some $r_{i+1}\in (\Z/p^\ell \Z)^\times$ we have
\begin{align*}f_i [\kappa_{\bold{\alpha}}(\mf{a}_{i+1})^\chi]_{\mf{Q}_{i+1}}=r_{i+1}p^{3e} f_i \phi_i(\kappa_{\bold{\alpha}}(\mf{a}_i)^\chi)\mf{Q}_{i+1}=r_{i+1}p^{4e} \lambda^2 \bar{v}_{\mf{Q}_i}(\kappa_{\bold{\alpha}}(\mf{a}_i)^\chi)\mf{Q}_{i+1},
\end{align*}
where the last equation follows from \eqref{eq15}. This proves \eqref{eq13} for $i+1$.
Finally, combining \eqref{eq13} for $1\leqs i\leqs k+1$ gives
\[\prod_{i=1}^k f_i \bar{v}_{\mf{Q}_{k+1}}(\kappa_{\bold{\alpha}}(\mf{a}_{k+1})^\chi)=r p^{(4k+4)e} \lambda^{2k+2} \beta\]
in $\Lambda_n/p^\ell \Lambda_n$ for some $r\in (\Z/p^\ell \Z)^\times$. It follows that 
\begin{small}
\[\mathrm{char}_\Lambda(A(H_\infty))=\prod_{i=1}^k f_i \text{\;\;  divides \;\;} p^{(4k+4)e} \lambda^{2k+2} \beta\Lambda(\ms{G})= p^{(4k+4)e} \lambda^{2k+2} \mathrm{char}_\Lambda\left(\bar{\mc{E}}_{H_\infty}/\bar{\mc{C}}_{H_\infty}\right).\]
\end{small}
This holds for every $\lambda\in I(D_\p)\mc{B}$ (or $2 I(\ms{G})\mc{B}$ if $p=2$), so in particular, holds for $\lambda$ being the greatest common divisor $\lambda_0$ of all elements in this ideal. If $p>2$, we have  $\lambda_0\Lambda(\ms{G})\supset I(D_\p)$ and the divisibility in (i) follows. If $p=2$, it is easy to show that we have $\lambda_0\Lambda(\ms{G})=2 I(\ms{G})$. This concludes the proof of Theorem \ref{thm9}, because $\mathrm{char}_\Lambda(A(H_\infty))$ is prime to $I(\ms{G})$ by Lemma \ref{thm12}.
\end{proof}

\begin{cor}\label{cor5}Let $p>2$. Then
\[\mathrm{char}_\Lambda(A(H_\infty))\text{\; divides\; }\mathrm{char}_\Lambda(\bar{\mc{E}}_{H_\infty}/ \bar{\mc{C}}_{H_\infty}).\]
\end{cor}

\begin{proof} We have shown in Lemma \ref{thm12} that $\mathrm{char}_\Lambda(A(H_\infty))$ is prime to $I(D_\mf{p})$, so by Theorem \ref{thm9}, $\mathrm{char}_\Lambda(A(H_\infty))$ divides $\mathrm{char}_\Lambda(\bar{\mc{E}}_{H_\infty}/ \bar{\mc{C}}_{H_\infty})$.
\end{proof}

Recall that $p\nmid [H:K]$ by assumption.

\begin{thm}\label{thm4.1} We have $\mathrm{char}_\Lambda(X(H_\infty))=\mathrm{char}_\Lambda\left(U_{H_\infty}/\bar{\mc{C}}_{H_\infty}\right)$ if and only if $\mathrm{char}_\Lambda(A(H_\infty))=\mathrm{char}_\Lambda\left(\bar{\mc{E}}_{H_\infty}/\bar{\mc{C}}_{H_\infty}\right)$, and 
\[\mathrm{char}_\Lambda(X(H_\infty))\mid 2^{e(6k+6)}\mathrm{char}_\Lambda\left(U_{H_\infty}/\bar{\mc{C}}_{H_\infty}\right).\]
\end{thm}

\begin{proof} Recall from \eqref{eq4.2} that we have an exact sequence
\[0\to \bar{\mc{E}}_{H_\infty}/\bar{\mc{C}}_{H_\infty}\to U_{H_\infty}/\bar{\mc{C}}_{H_\infty}\to X(H_\infty)\to A(H_\infty)\to 0,\]
and therefore $\mathrm{char}_\Lambda(A(H_\infty))\mathrm{char}_\Lambda\left(U_{H_\infty}/\bar{\mc{C}}_{H_\infty}\right)=\mathrm{char}_\Lambda(X(H_\infty))\mathrm{char}_\Lambda\left(\bar{\mc{E}}_{H_\infty}/\bar{\mc{C}}_{H_\infty}\right)$. The last assertion of the theorem follows from Theorem \ref{thm9} and Corollary \ref{cor5}.
\end{proof}~

\section{Proof of the Main Conjecture for $H_\infty/H$}\label{ch6}

\subsection{The Iwasawa Invariants of $X(H_\infty)$}\label{section6.2}~
\vspace{5 pt}\\
Recall that $\ms{G}\simeq G\times \Gamma$, where we identify $G$ with $\Gal(H_\infty/K_\infty)$ and $\Gamma$ with $\Gal(K_\infty/K)$. Let us consider $\Lambda_\ms{I}(\Gamma)$ as a $\Lambda_\ms{I}(\ms{G})$-module via $\chi\in G^*$. Given a finitely generated torsion $\Lambda_\ms{I}(\ms{G})$-module $M$, recall that $\mathrm{char}(M)\sb \Lambda_\ms{I}(\ms{G})$ denotes the characteristic ideal of $M$. If $X$ is a $\Lambda(\ms{G})$-module and $\chi\in G^*$, we write $X^\chi$ for $e_\chi(X\widehat{\otimes} _{\Z_p} \ms{I})$, where $e_\chi$ is the idempotent corresponding to $\chi$. This is justified because the characteristic ideal of a $\Gamma$-module behaves w~
\vspace{5 pt}\\ell under extension of scalars. This comes from the fact that we can identify $\Lambda_\ms{I}(\Gamma)$ with $\ms{I}[[T]]$ upon fixing a topological generator $\gamma$ of $\Gamma$.

Recall also that any $f(T)\in \ms{I}[[T]]$ can be written uniquely, by the $p$-adic Weierstrass preparation theorem, in the form
\[f(T)=\bold{\pi}^m P(T)U(T)\]
where $\bold{\pi}$ is a uniformiser of $\ms{I}$, $P(T)$ is a distinguished polynomial, that is, a monic polynomial whose coefficients are divisible by $\bold{\pi}$, and $U(T)$ is a unit in $\ms{I}[[T]]$. Let $\epsilon$ be the absolute ramification index of $\ms{I}$. The invariants $\mu(f)=\frac{m}{\epsilon}$ and $\lambda(f)=\deg P(T)$ are called the Iwasawa $\mu$-invariant and $\lambda$-invariant of $f$, respectively. The Iwasawa invariants of $\Lambda(\ms{G})$-modules are defined similarly, and if $M=X\widehat{\otimes}_{\Z_p}\ms{I}$ is obtained from a $\Lambda(\ms{G})$-module $X$ by extension of scalars to $\ms{I}$, the invariants of $M$ coincide with those of $X$. Let $f^\chi$ be a generator of $\mathrm{char}\left(X(H_\infty)^\chi\right)$ and let $g^\chi$ be a generator of $\mathrm{char}\left((U_{H\infty}/\bar{\mc{C}}_{H_\infty})^\chi\right)$. We set $f=\prod f^\chi$ and $g=\prod g^\chi$. By Theorem \ref{thm4.1}, we have 

\begin{thm} $f^\chi \mid \bold{\pi}^{ek} g^\chi$ for some integer $k\geq 0$, $e=0$ if $p>2$ and $e=1$ if $p=2$.
\end{thm}

In order to prove Theorem \ref{mc}, we just need to show $f^\chi$ and $g^\chi$ define the same ideal, thanks to Corollary \ref{cor3.1}. Thus it remains to show that $f$ and $g$ the have the same Iwasawa invariants. We shall compute them separately, and show that they are equal. First, we compute the invariants of $X(H_\infty)$ using class field theory, and in Section \ref{section6.3} we compute the invariants of $U_{H_\infty}/\bar{\mc{C}}_{H_\infty}$ using the analytic class number formula and Kronecker's second limit formula.

For the rest of the chapter, fix an integer $n\geqs 1+e$, where $e=0$ or $1$ if $p$ is odd or even. Recall from the proof of Proposition \ref{thm11} that $X(H_\infty)/I(H_n)X(H_\infty)$ is equal to $\Gal(M(H_n)/H_\infty)$, where $M(H_n)$ is the maximal abelian $p$-extension of $H_n$ which is unramified outside the primes of $H_n$ above $\mf{p}$. Thus the asymptotic formula of Iwasawa \cite[Theorem 13.13]{Was} gives

\begin{thm}\label{thm6.1}Let $f$ be a characteristic power series for $X(H_\infty)$ as a $\Z_p[[\Gamma]]$-module. For sufficiently large n, we have
\[\ord_p\left(\#(X(H_\infty)/I(H_n) X(H_\infty)\right)=\mu(f) p^{n-1-e}+\lambda(f) (n-1-e) +c,\]
where $\mu(f)$ and $\lambda(f)$ are the Iwasawa invariants of $X(H_\infty)$ and $c\in \Z$ is independent of $n$.
\end{thm}

We will now compute the $p$-adic valuation of the index $[M(H_n):H_\infty]$  using the methods of Coates and Wiles \cite{coa-wil2}, and use it to find $\ord_p\left(\#(X(H_\infty)/I(H_n) X(H_\infty)\right)$. We note that $p$ is assumed to be an odd prime number in \cite{coa-wil2}, but it can be extended to $p=2$ in our case, using the fact that $2$ splits in $K$ and $(p,h)=1$ by assumption. Set $[H_n:K]=d$, and recall that this is equal to $p^{n-1-e}h$. Let $\xi_1,\ldots \xi_d$ denote the distinct embeddings of $H_n$ into $\ovl{\Q}_p$ extending the embeddings of $K$ into $\Q_p$ given by $\mf{p}$. Since $H_n$ is totally imaginary, $\rank_\Z\left(\mc{E}_{H_n}/(\mc{E}_{H_n})_{tor}\right)=d-1$. We pick a basis $\epsilon_1, \ldots ,\epsilon_{d-1}$ for $\mc{E}_{H_n}/(\mc{E}_{H_n})_{tor}$, and put $\epsilon_d=1+p$ or $1+p^2$ if $p$ is odd or even. We define the $\mathfrak{p}$-adic regulator of Leopoldt:
\begin{equation}\label{regulator}R_{\mf{p}}(H_n)=(d\log \epsilon_d)^{-1}\det \left(\log (\xi_i(\epsilon_j))\right)_{1\leqs i, j\leqs d}.
\end{equation}
For each $n\geqs 1+e$, let $C_{H_n}$ denote the idele class group of $H_n$, and put
\[Y_n=\cap_{m\geqs n}N_{H_m/H_n}C_{H_m}.\]
We write $\Phi_\mf{p}=H_n\otimes_K K_\mf{p}$, and let $\ms{P}$ denote the set of primes of $H_n$ lying above $\p$. For each $\mf{P}\in \ms{P}$, let $U_{H_{n,\mf{P}}}$ be the group of units in the completion of $H_n$ at $\mf{P}$ which are congruent to $1$ modulo $\mf{P}$, and let $t\geqs 0$ be such that $p^{-t}\o_{\mf{P}}\sb \log U_{H_{n,\mf{P}}}$. The $p$-adic logarithm gives a homomorphism $\log: U_{H_{n,\mf{P}}}\to H_{n,\mf{P}}$ whose kernel has order $w_\mf{P}=\#\bold{\mu}_{p^\infty}(H_{n,\mf{P}})$. We write $\log U_{H_n}=\prod_{\mf{P}\in \ms{P}}\log U_{H_{n,\mf{P}}}$, so that we have $\log : U_{H_n}\to \Phi_\mf{p}$ with kernel $w_\mf{p}=\prod_{\mf{P}\in \ms{P}}w_\mf{P}$. Then the arguments of \cite[Lemma 7]{coa-wil2} apply, and we have

\begin{equation}\label{lem6.1}
\ord_p\left([\prod_{\mf{P}\in \ms{P}}p^{-t}\o_\mf{P}: \log U_{H_n}]\right)=\ord_p\left(w_\mf{p} \prod_{\mf{P}\in \ms{P}}\mathrm{N}\mf{P}\right)+td.
\end{equation}

Let $V=1+\mathfrak{p}\mathcal{O}_\mathfrak{p}$. We have $V^{p^n}=1+\mathfrak{p}^{n+1}\mathcal{O}_\mathfrak{p}$, the local units of $K_\mathfrak{p}$ which are congruent to $1$ modulo $\mathfrak{p}^{n+1}$. Note that $V=\{\pm 1\}V^{2}$ when $p=2$. We define $D_n=V^{p^e}\bar{\mathcal{E}}_{H_n}\subset U_{H_n}$, and set $V^{(n)}=V^{p^{n-1}}$ if $p>2$, and $V^{(n)}=\{\pm 1\}V^{p^{n-1}}$ if $p=2$. Furthermore, let $\Delta_{H_n/K}$ denote the discriminant of $H_n/K$, and pick a generator $\Delta_n$ of the ideal $\Delta_{H_n/K}\mathcal{O}_\mathfrak{p}$.

\begin{lem}\label{lem6.2}
\[\ord_p\left([\log U_{H_n}: \log D_n]\right)=\ord_p\left(\frac{R_{\mf{p}}(H_n)}{\sqrt{\Delta_n}}\left(w_\mathfrak{p}\prod_{\mathfrak{P}\in \mathscr{P}}\mathrm{N}\mathfrak{P}\right)^{-1}\right)+n.\]
\end{lem}

\begin{proof} Using methods analogous to \cite[Lemma 8]{coa-wil2}, we can show that 
\begin{equation}\label{eq6.1}\ord_p\left([\prod_{\mathfrak{P}\in \mathscr{P}}p^{-t}\mathcal{O}_\mathfrak{P}: \log D_n]\right)=\ord_p\left(\frac{R_{\mf{p}}(H_n)}{\sqrt{\Delta_n}}\right)+td+n-1-e+\ord_p\left(\log(\epsilon_d)\right).
\end{equation}
We have $\ord_p\left(\log(\epsilon_d)\right)=1+e$, so the right hand side of \eqref{eq6.1} is equal to $\ord_p\left(\frac{R_{\mf{p}}(H_n)}{\sqrt{\Delta_n}}\right)+td+n$. The result now follows from \eqref{lem6.1}.
\end{proof}

\begin{cor}\label{cor5.4}
\[\ord_p\left([U_{H_n}: D_n]\right)=\ord_p\left(\frac{R_{\mf{p}}(H_n)}{\omega_{H_n}\sqrt{\Delta_n}}\left(\prod_{\mathfrak{P}\in \mathscr{P}}\mathrm{N}\mathfrak{P}\right)^{-1}\right)+n.\]
\end{cor}

\begin{proof}This is an immediate consequence of Lemma \ref{lem6.2}, obtained by applying the snake lemma to the following commutative diagram 
\[
\begin{tikzcd}
0\arrow{r}& D_n \arrow{r} \arrow[swap]{d}{\log} & U_{H_n} \arrow{r} \arrow{d}{\log} & U_{H_n}/D_n\arrow{r}\arrow[swap]{d}&0\\
0\arrow{r}&\log D_n \arrow{r} & \log U_{H_n}\arrow{r}& \log U_{H_n}/\log D_n\arrow{r}&0.
\end{tikzcd}
\]
with exact rows.
\end{proof}

Following the arguments of \cite[Lemma 5]{coa-wil2}, we see that for $p>2$,
\begin{align*}Y_n\cap U_{H_n}&=\ker\left(\mathrm{N}_{\Phi_\mathfrak{p}/K_\mathfrak{p}}\mid_{U_{H_n}}\right),
\end{align*}
where $\mathrm{N}_{\Phi_\mathfrak{p}/K_\mathfrak{p}}\mid_{U_{H_n}}$ denotes the product of the local norms from $U_{H_n,\mathfrak{P}}$ to $K_\mathfrak{p}$ at each prime $\mathfrak{P}\in \mathscr{P}$. 
For $p=2$, we can make the following modification, as pointed out to me by Jianing Li. We claim that 
$Y_n\cap U_{H_n}$ is the inverse image of $\{\pm 1\}$ under $N_{\Phi_\mathfrak{p}/K_\mathfrak{p}}\mid_{U_{H_n}}$. Indeed, if we write $\Psi_\infty$ for its completion of $K_\infty$ at the unique prime above $\p$ and write $\Psi_n$ for the $n$-th layer, $\Psi_\infty/K_\mathfrak{p}$ is a totally ramified $\Z_2$-extension. Then by class field theory we have
$$V/\cap_{n} \mathrm{N}_n (V_n) \simeq \Gal(\Psi_\infty/K_\mathfrak{p}),$$
where $\mathrm{N}_n$ is the norm map from the principal units $V_n$ of $\Psi_n$ to $V$.
Since the right hand side is torsion-free and $V\simeq \Z_2^\times$, it follows that $-1\in \mathrm{N}_n(V_n)$ for any $n$. Furthermore, we have $V/\mathrm{N}_n(V_n)\simeq \Z/2^n\Z$, so $\mathrm{N}_n(V_n)=\{\pm 1\}V^{p^{n+1}}$. The rest of the argument follows from \cite[Lemma 5]{coa-wil2}.

We now claim that $\bar{\mathcal{E}}_{H_n}$ is contained in $Y_n\cap U_{H_n}$, and compute its index in the next lemma. Indeed, if $\xi\in \mathcal{E}_{H_n}$ and $p\neq 2$, then clearly $\mathrm{N}_{H_n/K}(\xi)=1$. If $p=2$, we have $\mathrm{N}_{H_n/K}(\xi)\in\{\pm 1\}$, so indeed $\xi$ is in the inverse image of $\{\pm 1\}$, as required. Now, we define 
\begin{align}\label{lem6.4}(Y_n\cap U_{H_n})^+=\ker\left(\mathrm{N}_{\Phi_\mathfrak{p}/K_\mathfrak{p}}\mid_{U_{H_n}}\right)
\end{align}
so that $(Y_n\cap U_{H_n})^+=Y_n\cap U_{H_n}$ if $p>2$ and $(Y_n\cap U_{H_n})^+$ has index $2$ in $Y_n\cap U_{H_n}$ if $p=2$.

\begin{lem}\label{lem6.3} 
\[[Y_n\cap U_{H_n}: \bar{\mathcal{E}}_{H_n}]=\ord_p\left(\frac{R_{\mf{p}}(H_n)}{\omega_{H_n}\sqrt{\Delta_n}}\prod_{\mathfrak{P}\in \mathscr{P}}(1-(\mathrm{N}\mathfrak{P})^{-1})\right)+n\]
\end{lem}

\begin{proof} By \eqref{lem6.4} and the definition of $D_n$, if we define $(\bar{\mathcal{E}}_{H_n})^+=\ker\left(\mathrm{N}_{\Phi_\mathfrak{p}/K_\mathfrak{p}}\mid_{D_n}\right)$ then $(\bar{\mathcal{E}}_{H_n})^+=\bar{\mathcal{E}}_{H_n}$ if $p>2$ and $(\bar{\mathcal{E}}_{H_n})^+$ has index $2$ in $\bar{\mathcal{E}}_{H_n}$ if $p=2$. Furthermore, since $(V^{p^{e}})^d=(V^{p^{e}})^{p^{n-1-e}}=V^{p^{n-1}}$, we see that $\mathrm{N}_{\Phi_\mathfrak{p}/K_\mathfrak{p}}(D_n)=V^{(n)}$ where we recall that $V^{(n)}=V^{p^{n-1}}$ if $p>2$, and $V^{(n)}=\{\pm 1\}V^{p^{n-1}}$ if $p=2$. Hence, we obtain  a commutative diagram with exact rows
\[
\begin{tikzcd}
0\arrow{r}&(\bar{\mathcal{E}}_{H_n})^+\arrow{r} \arrow[swap]{d} & D_n \arrow{r}{\mathrm{N}_{\Phi_\mathfrak{p}/K_\mathfrak{p}}} \arrow[swap]{d} & V^{(n)}\arrow{r}\arrow[swap]{d}&0\\
0\arrow{r}&(Y_n\cap U_{H_n})^+ \arrow{r} & U_{H_n}\arrow{r}{\mathrm{N}_{\Phi_\mathfrak{p}/K_\mathfrak{p}}}& V^{(n)}\arrow{r}&0.
\end{tikzcd}
\]
The lemma now follows from Corollary \ref{cor5.4} on noting that $[(Y_n\cap U_{H_n})^+: (\bar{\mathcal{E}}_{H_n})^+]=[Y_n\cap U_{H_n}: \bar{\mathcal{E}}_{H_n}]$ and $\ord_\mathfrak{p}\left(\prod_{\mathfrak{P}\in \mathscr{P}}(1-(\mathrm{N}\mathfrak{P})^{-1})\right)=\ord_p\left(\prod_{\mathfrak{P}\in \mathscr{P}}(\mathrm{N}\mathfrak{P})^{-1})\right)$.
\end{proof}

\begin{thm}\label{thm15}Let $M(H_n)$ be the maximal abelian $p$-extension of $H_n$ which is unramified outside the primes in $\mathscr{P}$, and write $h_{H_n}$ for the class number of $H_n$. Then
\[\mathrm{ord}_p\left([M(H_n):H_\infty]\right)=\ord_p\left(\frac{h_{H_n}R_{\mf{p}}(H_n)}{\omega_{H_n}\sqrt{\Delta_n}}\prod_{\mathfrak{P}\in \mathscr{P}}(1-(\mathrm{N}\mathfrak{P})^{-1})\right)+n.\]
\end{thm}

\begin{proof}Let $L(H_n)$ be the maximal unramified extension of $H_n$ in $M(H_n)$. Thus we may identify $\Gal(L(H_n)/H_n)$ with $A(H_n)$, the $p$-primary part of the ideal class group of $H_n$. Class field theory gives an isomorphism
\[Y_n\cap U_{H_n}/\bar{\mathcal{E}}_{H_n}\cong \Gal(M(H_n)/L(H_n)H_\infty).\] 
Note that $L(H_n)\cap H_\infty=H_n$ because $H_\infty/H_n$ is totally ramified at the primes in $\mathscr{P}$. Thus
\[0\to Y_n\cap U_{H_n}/\bar{\mathcal{E}}_{H_n}\to \Gal(M(H_n)/H_\infty)\to A(H_n)\to 0.\]
The theorem now follows from Lemma \ref{lem6.3} and the fact that $\ord_p\left(\#(A(H_n))\right)=\ord_p\left(h_{H_n}\right)$.
\end{proof}

\begin{cor}\label{cor4.1}Let $f$ be a characteristic power series for $X(H_\infty)$ as a $\Gamma$-module. Then for sufficiently large $n$,
\[\mu(f)\cdot p^{n-1-e}(p-1)+\lambda(f)=1+\ord_p\left(\frac{h_{H_{n+1}}R_{\mf{p}}(H_{n+1})}{\sqrt{\Delta_{n+1}}}/\frac{h_{H_n}R_{\mf{p}}(H_n)}{\sqrt{\Delta_n}}\right),\]
where $\mu(f)$ and $\lambda(f)$ are the Iwasawa invariants of $X(H_\infty)$.
\end{cor}

\begin{proof} Noting that each $\mathfrak{P}\in \mathscr{P}$ has inertia degree zero in $H_{n+1}/H_n$, it is clear from Theorem \ref{thm15} that the right hand side of the above equation is equal to $\ord_p\left([M(H_{n+1}):H_\infty]/[M(H_n):H_\infty]\right)$. Recalling that $\Gal(M(H_n)/H_\infty)=\linebreak X(H_\infty)/I(H_n)X(H_\infty)$, Theorem \ref{thm6.1} gives $\ord_p\left([M(H_{n+1}):H_\infty]/[M(H_n):H_\infty]\right)$ is equal to
\begin{small}
\[(\mu(f) p^{n-e}+\lambda(f) (n-e)+c)-(\mu(f) p^{n-1-e}+\lambda(f) (n-1-e)+c)=\mu(f) p^{n-1-e}(p-1)+\lambda(f).\]
\end{small}
This completes the proof of the corollary.
\end{proof}~

\subsection{The Iwasawa Invariants of the $\mf{p}$-adic $L$-function}\label{section6.3}~
\vspace{5 pt}\\
In this section, we will compute the Iwasawa invariants of $U_{H_\infty}/\bar{\mc{C}}_{H_\infty}$ and show that they coinside with those of $X(H_\infty)$ computed in Corollary \ref{cor4.1}. We will follow the methods discussed in \cite[Chapter III.2]{dS}. Again, the prime $p$ is assumed to be odd in Chapter III of \cite{dS}, but the methods still holds for $p=2$ thanks to our assumptions that $p$ splits in $K$ and $p\nmid [H:K]$. Fix a generator $g^\chi\in \Lambda_\ms{I}(\Gamma)$ of $\mathrm{char}\left((U_{H_\infty}/\bar{\mc{C}}_{H_\infty})^\chi\right)$, and set $g=\prod g^\chi$. Let $\mu(g)$ and $\lambda(g)$ denote the Iwasawa invariants of $g$. The following is \cite[Lemma III.2.9]{dS}. 

\begin{lem}\label{lem5.1}Recall that $\Gamma_n=\Gamma^{p^{n-1-e}}$. For any character $\rho$ of $\Gamma$ of finite order, write $l(\rho)=n-1-e$ if $\rho(\Gamma_n)=1$ but $\rho(\Gamma_{n-1})\neq 1$. Then for $n$ sufficiently large,
\[\ord_p\left(\prod_{l(\rho)=n-e}\rho(g)\right)=\mu(g)\cdot p^{n-1-e}(p-1)+\lambda(g).\]
\end{lem}

Given a ramified character $\varepsilon$ of $\ms{G}=G\times \Gamma$, write $\varepsilon=\chi\rho$ where $\chi$ is a character of $G$ and $\rho$ is a character of $\Gamma$.
 Let $\mf{f}_\varepsilon$ denote the conductor of $\varepsilon$, $f_\varepsilon=\mf{f}_\varepsilon\cap \Z$, and let $B_n$ be the collection of all $\varepsilon$ with $\mf{p}^n\mid\mid\mf{f}_\varepsilon$. Then

\begin{thm}\label{thm5.1}For $n$ sufficiently large,
\[ord_p\left(\prod_{l(\rho)=n-e}\rho(g)\right)=1+\ord_p\left(\frac{h_{H_{n+1}}R_{\mf{p}}(H_{n+1})}{\sqrt{\Delta_{n+1}}}/\frac{h_{H_n}R_{\mf{p}}(H_n)}{\sqrt{\Delta_n}}\right).\]
\end{thm}

\begin{proof} Our arguments are analogous to Proposition III.2.10 and III.2.11 of \cite{dS}. Any $\varepsilon\in B_{n}$ can be written in the form $\varepsilon=\chi\rho$ where $\chi$ is a character of $G$ and $\rho$ is a character of $\Gamma$ with $l(\rho)=n$. Define $G(\varepsilon)$ as in \cite[II.4.11]{dS} and $S_p(\varepsilon)$ be as defined in \cite[III.2.10 (14), III.2.11 (14')]{dS}. 
Recall that $g^\chi$ is also a generator of $\bold{\varphi^\chi}$ by Corollary \ref{cor3.1}, where $\bold{\varphi}=I_{\ms{I}}(\ms{G})\nu_\mf{p}$ and $\nu_\mf{p}$ satisfies Theorem \ref{thm4.2.7}. Then following the methods of \cite[Theorem II.5.2]{dS} but using $R_{\mc{D}}(\Phi(z,L))=\prod_{i=1}^r R_{\mathfrak{a}_i}(\Phi(z,L))^{n_i}$ constructed in Chapter \ref{ch4} instead of the elliptic function $\Theta(z;L,\mathfrak{a}_i)$ used in \cite{dS}, and noting $R_{\mathfrak{a}_i}(\Phi(z,L))$ is a $12$th root of $\Theta(z;L,\mathfrak{a}_i)$ and that $\theta_\mc{D}\nu_\mf{p}=\nu_\mc{D}$, we obtain
\[
\rho(g^\chi)=\int_{\ms{G}^\chi}\chi\rho d\nu_\mf{p}^\chi= \left\{ \begin{array}{ll}
 G(\varepsilon)S_p(\varepsilon) &\mbox{ if $\chi$ is non-trivial} \\
 \left(\rho(\gamma)-1\right)G(\varepsilon)S_p(\varepsilon) &\mbox{ if $\chi=1$,}
       \end{array} \right.
\]
where $\gamma$ is a topological generator of $\Gamma$. Hence $\prod_{l(\rho)=n-e}\left(\rho(\gamma)-1\right)=\prod_{\zeta\in \bold{\mu}_{p^{n-e}}}(\zeta-1)$. Noting that $\ord_p\left(\prod_{\zeta\in \bold{\mu}_{p^{n-e}}}(\zeta-1)\right))=1$, we obtain
\begin{equation}\label{eq6.2}ord_p\left(\prod_{l(\rho)=n-e}\rho(g)\right)=1+\ord_p\left(\prod_{\varepsilon\in B_{n+1-e}}G(\varepsilon)S_p(\varepsilon)\right).
\end{equation}
On the other hand, using the analytic class number formula and Kronecker's second limit formula (see \cite[\S 0.2.7, \S I.2.2 and \S IV.3.9 (6)]{Ta}) for the fields $H_{n+1}$ and $H_n$ gives:
\begin{equation}\label{eq6.3}\ord_p\left(\prod_{\varepsilon\in B_{n+1-e}}G(\varepsilon)S_p(\varepsilon)\right)=\ord_p\left(\frac{h_{H_{n+1}}R_{\mf{p}}(H_{n+1})}{\sqrt{\Delta_{n+1}}}/\frac{h_{H_n}R_{\mf{p}}(H_n)}{\sqrt{\Delta_n}}\right)
\end{equation}
Combining \eqref{eq6.2} and \eqref{eq6.3} completes the proof of Proposition \ref{thm5.1}.
\end{proof}

Comparing Corollary \ref{cor4.1}, Lemma \ref{lem5.1} and Theorem \ref{thm5.1}, we conclude that $f$ and $g$ have the same Iwasawa invariants. As discussed at the beginning of Section \ref{section6.2}, this together with the divisibility relation obtained in Theorem \ref{thm4.1} completes the proof of Theorem \ref{mc}. 

 \section*{acknowledgement}
I would like to express my sincere gratitude to John Coates for patiently advising and encouraging me. His expertise and enthusiasm helped me overcome many difficulties. I would also like to extend my heartfelt thanks to Junhwa Choi, Tom Fisher, Guido Kings, Jack Lamplugh, Jianing Li and Sarah Zerbes for giving me many helpful comments on this work. This work was partially supported by the SFB 1085 ``Higher invariants'' at the University of Regensburg and I would also like to thank the Max Planck Institute for Mathematics in Bonn for the support and hospitality provided by the institute. Finally, I am very grateful to the referees for their very careful reading of the paper and for the valuable feedback which helped improving the quality of this paper.

\vspace{10pt}
\noindent \author{Yukako Kezuka}\\
\address{Max-Planck-Institut f\"{u}r Mathematik\\
Bonn, Germany}\\
\email{ykezuka@mpim-bonn.mpg.de}

\end{document}